\documentclass[12pt]{article}
\usepackage{graphicx}
\usepackage{amsmath,amsthm,amssymb,amsfonts,euscript,enumerate}
\usepackage{amsthm}
\usepackage{color,wrapfig}
\usepackage{pxfonts}
\usepackage{verbatim}
\newtheorem{thm}{Theorem}[section]
\newtheorem{cor}[thm]{Corollary}
\newtheorem{lem}[thm]{Lemma}
\newtheorem{prop}[thm]{Proposition}

\newcommand{\1}{\partial}

\newcommand{\3}{\varepsilon}

\newcommand{\Z}{{\mathbb Z}}

\newcommand{\cD}{{\mathcal D}}

\newcommand{\cX}{{\mathcal X}}

\def\ni{\noindent}

\begin{document}
\title{Existence and uniqueness  of the singular self-similar solutions of the fast diffusion equation and logarithmic diffusion equation}
\author{Kin Ming Hui\\
Institute of Mathematics, Academia Sinica\\
Taipei, Taiwan, R. O. C.\\
e-mail: kmhui@gate.sinica.edu.tw}
\date{Dec 31, 2024}
\smallbreak \maketitle
\begin{abstract}
Let $n\ge 3$, $0<m<\frac{n-2}{n}$, $\rho_1>0$,  $\eta>0$, $\beta>\frac{m\rho_1}{n-2-nm}$, 
$\alpha=\alpha_m=\frac{2\beta+\rho_1}{1-m}$, $\beta_0>0$ and $\alpha_0=2\beta_0+1$. We use fixed point argument to give a new proof for the existence and uniqueness of radially symmetric singular solution $f=f^{(m)}$ of the  elliptic equation $\Delta (f^m/m)+\alpha f+\beta x\cdot\nabla f=0$, $f>0$, in $\mathbb{R}^n\setminus\{0\}$, satisfying $\displaystyle\lim_{|x|\to 0}|x|^{\alpha/\beta}f(x)=\eta$.
We also prove the existence and uniqueness of  radially symmetric singular solution $g$ of the  equation $\Delta\log g+\alpha_0 g+\beta_0x\cdot\nabla g=0$, $g>0$, in $\mathbb{R}^n\setminus\{0\}$, satisfying $\displaystyle\lim_{|x|\to 0}|x|^{\alpha_0/\beta_0}g(x)=\eta$.  
Such equations arises from the study of backward singular self-similar solution of the fast diffusion equation $u_t=\Delta u^m$ and the logarithmic diffusion equation $u_t=\Delta\log u$ respectively. 
We will also prove the asymptotic decay rate of the function $f$ as $|x|\to\infty$.

\end{abstract}

\vskip 0.2truein



\vskip 0.2truein
\setcounter{equation}{0}
\setcounter{section}{0}

\section{Introduction}
\setcounter{equation}{0}
\setcounter{thm}{0}

Recently there is a lot of study on the equation
\begin{equation}\label{fde}
u_t=\Delta (u^m/m)
\end{equation}
for the range $0<m<\frac{n-2}{n}$ by P.~Daskalopoulos, M.~Fila, S.Y.~Hsu, K.M.~Hui, T.~Jin, S.~Kim, P.~Mackov\'a, M.~del Pino, J.~King, N.~Sesum, Y.~Sire, J.~Takahashi, J.L.~Vazquez, J.~Wei, M.~Winkler, J.~Xiong, H.~Yamamoto,  E.~Yanagida and Y.~Zheng, \cite{DKS}, \cite{DPS}, \cite{DS1}, \cite{DS2}, \cite{FMTY}, \cite{FVWY}, \cite{FW}, \cite{Hs2}, \cite{Hs3}, \cite{Hs4}, \cite{Hu1}, \cite{Hu2}, \cite{Hu3}, \cite{Hu4}, \cite{HuK2}, \cite{HuK3}, \cite{JX}, \cite{SWZ}, \cite{TY}, etc. Equation \eqref{fde} appears in many physical models and in geometric flow problem. 
When $m>1$, \eqref{fde} is the porous medium equation which models the flow of gases or liquid through porous media \cite{A}. When $0<m<1$, \eqref{fde} is the fast diffusion equation \cite{V}. When $m=\frac{n-2}{n+2}$, $n\ge 3$, and $g=u^{\frac{4}{n+2}}dx^2$ is a metric on $\mathbb{R}^n$ which evolves by the Yamabe flow,
\begin{equation*}
\frac{\1 g}{\1 t}=-Rg
\end{equation*}
where $R$ is the scalar curvature of the metric $g$, then $u$ satisfies (\cite{DKS}, \cite{PS}),
\begin{equation*}
u_t=\frac{n-1}{m}\Delta u^m
\end{equation*}
which after rescaling is equivalent to \eqref{fde}. Letting $m\to 0^+$ in \eqref{fde}, we formally get
 the logarithmic diffusion equation (\cite{Hu2}, \cite{HuK1}),
\begin{equation}\label{log-diffusion-eqn}
u_t=\Delta\log u.
\end{equation}
This equation arises  in the Carleman model of two types of particles moving with velocities $\pm 1$ along the $x$-axis ($n=1$)  and in the study of Ricci flow on surface ($n=2$) etc. We refer the readers to the book \cite{DK} by P.~Daskalopoulos and C.E.~Kenig and the book [V] by J.L.~Vazquez for some of the recent results on the equation \eqref{fde} and \eqref{log-diffusion-eqn}.

It is known that under some appropriate conditions on the initial data the solutions of a partial differential equation will behave like the self-similar solution of the equation. Hence in order to study the behaviour of the singular solutions of \eqref{fde} and \eqref{log-diffusion-eqn} in $\mathbb{R}^n\setminus\{0\}$ which blow up at the origin it is important to study the self-similar singular solutions of \eqref{fde} and \eqref{log-diffusion-eqn} in $\mathbb{R}^n\setminus\{0\}$ which blow up at the origin. Now the backward self-similar solution of \eqref{fde} can be written as 
\begin{equation}\label{sss}
V(x,t)=(T-t)^{\alpha}f((T-t)^{\beta}x),\quad x\in\mathbb{R}^n\setminus\{0\}, t<T
\end{equation}
for some constant $\alpha,\beta$. Note that for any $0<m<1$, $V$ is a solution of \eqref{fde} in $\mathbb{R}^n\setminus\{0\}\times (-\infty,T)$ if  $f$ satisfies
\begin{equation}\label{elliptic-eqn}
\Delta (f^m/m)+\alpha f+\beta x\cdot\nabla f=0,\quad f>0, \quad\mbox{ in }\mathbb{R}^n\setminus\{0\}
\end{equation}
and
\begin{equation}\label{alpha-beta-relation}
\alpha=\alpha_m=\frac{2\beta+1}{1-m}.
\end{equation}
For any $n\ge 3$, $0<m<\frac{n-2}{n}$, $\eta>0$, $\rho_1>0$, $\beta>\frac{m\rho_1}{n-2-nm}$ and
\begin{equation}\label{alpha-beta-relation2}
\alpha=\frac{2\beta+\rho_1}{1-m}, 
\end{equation}
existence of radially symmetric solution $f$ of \eqref{elliptic-eqn} which satisfies
\begin{equation}\label{blow-up-rate-x=0}
\lim_{r\to 0}r^{\alpha/\beta}f(r)=\eta
\end{equation}
was proved by K.M.~Hui \cite{Hu3} using shooting method. Note that the shooting method is very hard for this problem since one need to find appropriate initial values so that the shooting method works. Uniqueness of such radially symmetric solution $f$ for the case $\beta\ge\frac{\rho_1}{n-2-nm}$ is also proved by K.M.~Hui \cite{Hu3}. The uniqueness result was extended to the case $\beta\ge\frac{m\rho_1}{n-2-nm}$  by K.M.~Hui and S.~Kim \cite{HuK3} by finding the asymptotic expansion of the solution $f$ near the origin and using integral comparison method.

Since the solution $f$ of \eqref{elliptic-eqn} which satisfies \eqref{blow-up-rate-x=0} blows up at the 
origin, one cannot apply fixed point argument directly to $f$ to prove the existence and uniqueness result. One need to transform the function $f$ into another function $z$ by \eqref{w-defn}
and \eqref{z-defn}. Then we  apply the
fixed point argument on a system of integral equations related to this function $z$ to give a new proof of the existence and uniqueness result for the case
$\beta>\frac{m\rho_1}{n-2-nm}$ with $\alpha$ being given by \eqref{alpha-beta-relation2}. 
Hence we give simple new proof of the results of \cite{Hu3} and \cite{HuK3}. 

This result says that for any $T>0$, $\eta>0$, $\alpha$ and $\beta>1/2$ satisfying \eqref{alpha-beta-relation} with $m=\frac{n-2}{n+2}$ there exists a unique radially symmetric Yamabe flow 
\begin{equation*}
g(t)=\left[(T-t)^{\alpha}f((T-t)^{\beta}x)\right]^{\frac{4}{n+2}}\,dx^2\quad\mbox{ on }(\mathbb{R}^n\setminus\{0\})\times (0,T)
\end{equation*} 
 where $dx^2$ is the Euclidean metrics on $\mathbb{R}^n$ and $f$ is the unique radially symmetric solution of 
\eqref{elliptic-eqn} which satisfies \eqref{blow-up-rate-x=0}.
The metric $g(t)$ vanishes at time $T$. Moreover for any compact subset $K\subset \mathbb{R}^n\setminus\{0\}$,  $g(t)\approx \eta^{-\frac{4}{n+2}} |x|^{-\frac{4\alpha}{(n+2)\beta}}\,dx^2$ on $K$ as $t\to T$.

Although the existence result of this paper is also mentioned in \cite{V}, there is no vigorous proof there and there is only a heuristic argument using the phase plane approach in \cite{V}.
Note that the existence result of \cite{Hu3} is proved by using shooting method with selection  of
the appropriate initial data for the approximate solutions. This method may not work in other problems since it is hard to select the appropriate initial data for similar proof to work.

On the other hand the uniqueness of radially symmetric solutions of \eqref{elliptic-eqn} which satisfies \eqref{blow-up-rate-x=0} is proved in \cite{HuK3} using a careful analysis of the asymptotic expansion of the solution 
$f$ near the origin. This method also may not work in other problems. Hence our approach in this paper
is more general and is applicable to a large number of problems.

Note that our method does not apply to the case $\beta=\frac{m\rho_1}{n-2-nm}$. The reason is that in this case the constants $C_2=0$ and $C_3=mC_1^2$ in \eqref{c123-defn}. Hence the constant $\3_1=0$ in \eqref{epsilon1-defn} so that the proof of Lemma \ref{z-local-existence-lem} does not work.
Our method also does not apply to the case $m<0$ since in this case the constant $\3_1$ in \eqref{epsilon1-defn} does not exists so that the proof of Lemma \ref{z-local-existence-lem} does not work.

The method developed here is applicable to the proof of existence and uniqueness of singular solutions for a large number of nonlinear equations. For example this method is used by K.M.~Hui and Jongmyeong Kim \cite{HuKj} to prove the existence and uniqueness of the forward singular self-similar solution of \eqref{fde} in $(\mathbb{R}^n\setminus\{0\})\times (0,\infty)$. This method unifies the proof of the existence and uniqueness of both the forward and backward singular self-similar solutions of \eqref{fde} in $(\mathbb{R}^n\setminus\{0\})\times (0,\infty)$ and $(\mathbb{R}^n\setminus\{0\})\times (0,T)$ for some constant $T>0$ respectively.
On the other hand in the paper \cite{Hu5} K.M.~Hui used this method to prove the existence of singular rotationally symmetric gradient Ricci solitons in 
higher dimensions which blows up at the origin.

The author suspects that $|x|^{-\alpha/\beta}$ is the only possible blow-up rate for this type of singular backward self-similar solution of \eqref{fde}. We would like to conjecture that there are no singular backward self-similar solution of \eqref{fde} with other blow-up rate.

On the other hand the backward self-similar solution of \eqref{log-diffusion-eqn} can be written as 
\begin{equation}\label{log-sss}
V_0(x,t)=(T-t)^{\alpha_0}g((T-t)^{\beta_0}x),\quad x\in\mathbb{R}^n\setminus\{0\}, t<T
\end{equation}
for some constants $\beta_0>0$ and $\alpha_0$. Note that  $V_0$ is a solution of \eqref{log-diffusion-eqn} in $\mathbb{R}^n\setminus\{0\}\times (-\infty,T)$ if  $g$ satisfies
\begin{equation}\label{log-elliptic-eqn}
\Delta \log g+\alpha_0 g+\beta_0 x\cdot\nabla g=0,\quad g>0, \quad\mbox{ in }\mathbb{R}^n\setminus\{0\}
\end{equation}
and
\begin{equation*}
\alpha_0=2\beta_0+1.
\end{equation*}
For any $n\ge 3$, $\eta>0$, $\beta_0>0$, $\rho_1>0$ and 
\begin{equation}\label{alpha0-beta0-relation}
\alpha_0=2\beta_0+\rho_1,
\end{equation}
K.M.~Hui and S.~Kim \cite{HuK3} proved
the existence of radially symmetric solutions $g$ of \eqref{log-elliptic-eqn} which satisfies 
\begin{equation}\label{g-blow-up-rate-x=0}
\lim_{r\to 0}r^{\alpha_0/\beta_0}g(r)=\eta
\end{equation}
by showing that the unique radially symmetric solution $f=f^{(m)}$ of \eqref{elliptic-eqn} with $\alpha=\alpha_m$ given by \eqref{alpha-beta-relation2} that satisfies \eqref{blow-up-rate-x=0}  converges on every compact subset of $\mathbb{R}^n\setminus\{0\}$ to the unique singular radially symmetric solution $g$ of \eqref{log-elliptic-eqn} which satisfies \eqref{g-blow-up-rate-x=0}  as $m\to 0^+$. In this paper we will use fixed point argument to give a new proof of the existence and uniqueness of such radially symmetric solution solution $g$ of \eqref{log-elliptic-eqn} with $\alpha_0$ given by \eqref{alpha0-beta0-relation} which satisfies \eqref{g-blow-up-rate-x=0}.

It is stated without proof in \cite{DKS} and \cite{Hu3} that the radially symmetric solution $f=f^{(m)}$ of \eqref{elliptic-eqn}, \eqref{blow-up-rate-x=0}, with $\alpha=\alpha_m$ given by \eqref{alpha-beta-relation2} satisfies
\begin{equation}\label{f-limit-infty}
\lim_{|x|\to\infty}|x|^2f(x)^{1-m}=\frac{2(n-2-nm)}{(1-m)\rho_1}.
\end{equation}
In this paper we will use a modification of the proof of \cite{Hs4} to prove this result.
More precisely we will prove the following main results.

\begin{thm}\label{blow-up-self-similar-soln-thm}
Let $n\ge 3$, $0<m<\frac{n-2}{n}$, $\rho_1>0$, $\eta>0$, $\beta>\frac{m\rho_1}{n-2-nm}$ and $\alpha$ satisfies \eqref{alpha-beta-relation2} . 
Then there exists a  unique   solution $f=f^{(m)}$ of 
\begin{equation}\label{f-ode}
(f^m/m)_{rr}+\frac{n-1}{r}(f^m/m)_r+\alpha f+\beta rf_r=0,\quad f>0, \quad\mbox{ in } (0,\infty)
\end{equation} 
in $C^2(0,\infty)$ that satisfies \eqref{blow-up-rate-x=0}. 
\end{thm} 

\begin{thm}\label{log-blow-up-self-similar-soln-thm}
Let $n\ge 3$, $\eta>0$, $\rho_1>0$, $\beta_0>0$  and $\alpha_0$ be given by \eqref{alpha0-beta0-relation}. Then there exists a unique solution $g$ of 
\begin{equation}\label{g-ode}
\left(\frac{g_r}{g}\right)_r+\frac{n-1}{r}\cdot\frac{g_r}{g}+\alpha_0 g+\beta_0 rg_r=0,\quad g>0, \quad\mbox{ in } (0,\infty)
\end{equation} 
in $C^2(0,\infty)$ that satisfies \eqref{g-blow-up-rate-x=0}.
\end{thm} 
 
\begin{thm}\label{soln-at-x-infty-thm}
Let $n\ge 3$, $0<m<\frac{n-2}{n}$, $\rho_1>0$, $\beta>\frac{m\rho_1}{n-2-nm}$ and $\alpha=\alpha_m$ be given by \eqref{alpha-beta-relation2}. Let $f\in C^2(0,\infty)$ be the 
unique  radially symmetric solution  of \eqref{elliptic-eqn} that satisfies \eqref{blow-up-rate-x=0}  given by Theorem \ref{blow-up-self-similar-soln-thm}. Then $f$ satisfies \eqref{f-limit-infty}.
\end{thm}

The plan of the paper is as follows. In section 2 we will prove Theorem \ref{blow-up-self-similar-soln-thm}. In section 3 we will prove Theorem \ref{log-blow-up-self-similar-soln-thm}. In section 4 we will prove Theorem \ref{soln-at-x-infty-thm}.

\section{Existence and uniqueness of the singular self-similar solution of the fast diffusion equation}
\setcounter{equation}{0}
\setcounter{thm}{0}

In this section we assume that $n\ge 3$, $0<m<\frac{n-2}{n}$, $\eta>0$, $\rho_1>0$, $\beta>0$ and $\alpha$  satisfies \eqref{alpha-beta-relation2}. We will  prove the existence and uniqueness of radially symmetric solution $f$ of \eqref{elliptic-eqn} which satisfies \eqref{blow-up-rate-x=0}. 
Suppose $f\in C^2(0,\infty)$ is a solution of \eqref{f-ode} which satisfies \eqref{blow-up-rate-x=0}. Let $w$  be given by 
\begin{equation}\label{w-defn}
w(r)=r^{\alpha/\beta}f(r)\quad\forall r=e^s,s\in\mathbb{R}
\end{equation}
and $z$ be given by
\begin{equation}\label{z-defn}
s=\log r,\quad z(s)=rw_r(r)/w(r)=\frac{\frac{\1}{\1 s}w(e^s)}{w(e^s)}.
\end{equation}
Then $w\in C^2(0,\infty)$ and $z\in C^1(\mathbb{R})$.
Note that the transforms $w$ and $z$ are not in the book \cite{V}. The transform $w$ first appears in \cite{Hs2} in the study of the asymptotic behaviour of the self-similar solution of \eqref{fde} in $\mathbb{R}^n$ as $|x|\to\infty$. On the other hand the transform $z$ first appear in \cite{Hu3}. 
Note that by direct computation  (cf. \cite{Hs2}) we have
\begin{equation}\label{w1-r-eqn}
\left(\frac{w_r}{w}\right)_r+\frac{n-1-\frac{2m\alpha}{\beta}}{r}\cdot\frac{w_r}{w}
+m\left(\frac{w_r}{w}\right)^2
+\frac{\beta r^{-1-\frac{\rho_1}{\beta}}w_r}{w^m}
=\frac{\alpha}{\beta}\cdot\frac{n-2-\frac{m\alpha}{\beta}}{r^2}\quad\forall r>0
\end{equation}
and $z$ satisfies
\begin{align}
&z_s+\left(n-2-\frac{2m\alpha}{\beta}\right)z+mz^2+\beta e^{-\frac{\rho_1}{\beta}s}\widetilde{w}^{1-m}z=\frac{\alpha}{\beta}\left(n-2-\frac{m\alpha}{\beta}\right)\label{z-eqn}\\
\Leftrightarrow\quad&z_s+m(z+C_1)^2+\beta e^{-\frac{\rho_1}{\beta}s}\widetilde{w}^{1-m}z=C_3\label{z-eqn2}
\end{align} 
in $\mathbb{R}$ where 
\begin{equation}\label{tilde-w-defn}
\widetilde{w}(s)=w(e^s)\quad\forall s\in\mathbb{R}
\end{equation}
and
\begin{equation}\label{c123-defn}
C_1=\frac{1}{2m}\left(n-2-\frac{2m\alpha}{\beta}\right),\quad C_2=\frac{\alpha}{\beta}\left(n-2-\frac{m\alpha}{\beta}\right)
\quad\mbox{ and }\quad C_3=C_2+mC_1^2.
\end{equation}
Note that when $\beta>\frac{m\rho_1}{n-2-nm}$, then $C_2>0$ and hence $C_3>0$. For any $b_1>0$ and function $w_1\in L_{loc}^{\infty}(\mathbb{R})$, let 
\begin{equation}\label{a-defn}
a(w_1,s)=-\beta\int_s^{-b_1}e^{-\frac{\rho_1}{\beta}\rho'}w_1^{1-m}(\rho')\,d\rho'\quad\forall s\in\mathbb{R}.
\end{equation}
If we assume that 
\begin{equation*}
z(s_1)e^{a(\widetilde{w},s_1)}\to 0\quad\mbox{ as }s_1\to -\infty
\end{equation*}
and
\begin{equation*}
\int_{-\infty}^se^{a(\widetilde{w},\rho)}\,d\rho+\int_{-\infty}^se^{a(\widetilde{w},\rho)}(z(\rho)+C_1))^2\,d\rho<\infty\quad\forall s\in\mathbb{R},
\end{equation*}
 then by multiplying  \eqref{z-eqn2} by $e^{a(\widetilde{w},s)}$ and integrating over $(-\infty,s)$ we get
\begin{equation}\label{z-representation-formula}
z(s)=C_3\int_{-\infty}^se^{a(\widetilde{w},\rho)-a(\widetilde{w},s)}\,d\rho -m\int_{-\infty}^se^{a(\widetilde{w},\rho)-a(\widetilde{w},s)}(z(\rho)+C_1))^2\,d\rho
\end{equation}
in $\mathbb{R}$.
This suggests one to use fixed point argument to prove the existence of radially symmetric solution $f$ of \eqref{elliptic-eqn} which satisfies \eqref{blow-up-rate-x=0}. 

We will prove Theorem \ref{blow-up-self-similar-soln-thm} in several steps. We will first use fixed point argument to prove the local existence and uniqueness of  solution $z$ of \eqref{z-eqn} 
with
\begin{equation}\label{tilde-w-formula} 
\widetilde{w}(s)=\eta\, \mbox{exp}\left(\int_{-\infty}^sz(\rho)\,d\rho\right)
\end{equation}
which satisfies some decay condition for $z$. We will then prove that this local solution can be extended to a unique global solution $z$ of \eqref{z-eqn} with $\widetilde{w}$ given by \eqref{tilde-w-formula} under the condition that $z$ satisfies some decay condition. We then remove the decay requirement on $z$ in the uniqueness of solution result. From these results on $z$ we get corresponding result for $f$ in Theorem \ref{blow-up-self-similar-soln-thm}.

\begin{lem}\label{z-local-existence-lem}
Let $n\ge 3$, $0<m<\frac{n-2}{n}$, $\rho_1>0$, $\eta>0$, $\beta>\frac{m\rho_1}{n-2-nm}$ and $\alpha$ satisfies \eqref{alpha-beta-relation2}. Then there exists a constant $b_1>0$ such that the equation \eqref{z-eqn}
has a unique  solution $z\in C^1(-\infty,-b_1)$   in $(-\infty,-b_1)$ 
with $\widetilde{w}$ given by \eqref{tilde-w-formula}
in $(-\infty, -b_1)$ which satisfies 
\begin{equation}\label{z-growth-rate-bd}
0<z(s)e^{-\frac{\rho_1}{\beta}s}\le C_4\eta^{m-1}
\end{equation} in $(-\infty, -b_1)$ 
with
\begin{equation}\label{c4-defn}
C_4=\frac{C_3}{\beta}
\end{equation}
where $C_3$ be given by \eqref{c123-defn}. 
\end{lem} 
\begin{proof}
We will prove this result by using fixed point argument. Let $\eta>0$ and $b_1>0$.
We define the Banach space 
\begin{equation*}
\cX_{b_1}=\left\{(\widetilde{w},z): \widetilde{w},z\in C((-\infty,-b_1); \mathbb{R})\,\,\mbox{ such that }\,\,\|(\widetilde{w},z)\|_{\cX_{b_1}}<\infty\right\}
\end{equation*}
with norm 
\begin{equation*}
\|(\widetilde{w},z)\|_{\cX_{b_1}}=\max\left\{\|\widetilde{w}\|_{L^\infty\left((-\infty,-b_1);e^{-\frac{\rho_1s}{4\beta}}\right)},\|z\|_{L^\infty\left((-\infty,-b_1);e^{-\frac{\rho_1s}{2\beta}}\right)}\right\}.
\end{equation*}
where
\begin{equation*}
\|v\|_{L^\infty\left((-\infty,-b_1);e^{-\lambda s}\right)}=\sup_{-\infty<s<-b_1}\left|v(s)e^{-\lambda s}\right|
\end{equation*}
for any $\lambda>0$. Since by the l'Hospital rule,
\begin{equation*}
\lim_{a\to 0}\frac{\mbox{exp}\,\left(\frac{C_3 a}{\rho_1\eta^{1-m}}\right)-1}{a\eta^{m-1}}=\frac{C_3}{\rho_1},
\end{equation*}
there exists a constant $C_5>0$ depending on $C_3$, $\eta$ and $\rho_1$ such that
\begin{equation}\label{exp-term-bd}
\left|\frac{\mbox{exp}\,\left(\frac{C_3 a}{\rho_1\eta^{1-m}}\right)-1}{a\eta^{m-1}}\right|\le C_5\quad\forall 0<|a|\le 1.
\end{equation}
Let 
\begin{equation*} 
\cD_{b_1}:=\left\{ (\widetilde{w},z)\in \cX_{b_1}:  \|(\widetilde{w},z)-(\eta,0)\|_{\cX_{b_1}}\le\3_1\mbox{ and }\widetilde{w}(s)\ge\eta, 0\le z(s)e^{-\frac{\rho_1}{\beta}s}\le C_4\eta^{m-1}\,\forall s<-b_1\right\}
\end{equation*}
where 
\begin{equation}\label{epsilon1-defn}
\3_1:=\frac{1}{2}\min\left(1,\eta,\sqrt{\frac{C_3}{m}}-|C_1|\right).
\end{equation}
 We now choose
\begin{equation}\label{b1-range}
b_1>b_0:=\max \left(1,\frac{2\beta}{\rho_1}\left|\log\left(\frac{C_4\eta^{m-1}}{\3_1}\right)\right|,\frac{4\beta}{3\rho_1}\log\left(\frac{C_5\eta^m}{\3_1}\right),\frac{\beta}{\rho_1}\log\left(\frac{3\rho_1}{4\beta^2\eta^{1-m}}\right)\right).
\end{equation}  
Then $(\eta,\min (C_4\eta^{m-1},\3_1)e^{\rho_1s/\beta})\in \cD_{b_1}$. Hence $\cD_{b_1}\ne\phi$.
For any $(\widetilde{w},z)\in \cD_{b_1},$ let 
\begin{equation*}
\Phi(\widetilde{w},z):=\left(\Phi_1(\widetilde{w},z),\Phi_2(\widetilde{w},z)\right)
\end{equation*}  
be given by 
\begin{equation}\label{phi12-defn}
\left\{\begin{aligned}
&\Phi_1(\widetilde{w},z)(s):=\eta\, \mbox{exp}\left(\int_{-\infty}^sz(\rho)\,d\rho\right)\qquad\qquad\qquad\qquad\qquad\qquad\qquad\qquad\forall s<-b_1\\
&\Phi_2(\widetilde{w},z)(s):=C_3\int_{-\infty}^se^{a(\widetilde{w},\rho)-a(\widetilde{w},s)}\,d\rho -m\int_{-\infty}^se^{a(\widetilde{w},\rho)-a(\widetilde{w},s)}(z(\rho)+C_1))^2\,d\rho\quad\forall s<-b_1
\end{aligned}\right.
\end{equation}
where $C_1$, $C_2$, $C_3$ and $a(\widetilde{w},s)$ are given by \eqref{c123-defn} and \eqref{a-defn} respectively. 
We first prove  that  $\Phi(\cD_{b_1})\subset \cD_{b_1}$. Let $(\widetilde{w},z)\in \cD_{b_1}$. Then $\widetilde{w}(s)\ge\eta$ for any $s<-b_1$. Since $b_1>0$ and $\beta>0$ we have
\begin{align*}
&\|\widetilde{w}(s)-\eta\|_{L^\infty\left((-\infty,-b_1);e^{-\frac{\rho_1s}{4\beta}}\right)}
\le\3_1\le\eta/2\quad\forall s<-b_1\\
\Rightarrow\quad&|\widetilde{w}(s)-\eta|\le\frac{\eta}{2}\mbox{exp}\,\left(\frac{\rho_1s}{4\beta}\right)\le\frac{\eta}{2}\mbox{exp}\,\left(-\frac{\rho_1b_1}{4\beta}\right)\le\frac{\eta}{2}\quad\forall 
s<-b_1.
\end{align*}
Hence
\begin{equation}\label{tilde-w-upper-lower-bd}
\eta\le\widetilde{w}(s)\le 3\eta/2\quad\forall s<-b_1
\end{equation}
and by the definition of $\Phi_1(\widetilde{w},z)$,
\begin{equation}\label{Phi1-ge-eta}
\Phi_1(\widetilde{w},z)(s)\ge\eta\quad\forall s<-b_1.
\end{equation}
Now by \eqref{phi12-defn}, 
\begin{align}\label{Phi2-positive}
\Phi_2(\widetilde{w},z)(s)\ge&\left(C_3-m(\|z\|_{L^{\infty}(-\infty,-b_1)}+|C_1|)^2\right)\int_{-\infty}^se^{a(\widetilde{w},\rho)-a(\widetilde{w},s)}\,d\rho\notag\\
\ge&\left(C_3-m(\3_1+|C_1|)^2\right)\int_{-\infty}^se^{a(\widetilde{w},\rho)-a(\widetilde{w},s)}\,d\rho\notag\\
>&0\qquad\qquad\qquad\qquad\forall s<-b_1.
\end{align}
By \eqref{a-defn}, \eqref{tilde-w-upper-lower-bd} and the mean value theorem for any $\rho<s<-b_1$ there exists a constant $\xi\in (\rho,s)$ such that
\begin{equation}\label{a-ineqn}
a(\widetilde{w},\rho)-a(\widetilde{w},s)=\beta e^{-\frac{\rho_1}{\beta}\xi}\widetilde{w}^{1-m}(\xi)(\rho-s)
\le\beta\eta^{1-m}e^{-\frac{\rho_1}{\beta}s}(\rho-s)\quad\forall \rho<s<-b_1.
\end{equation}
Hence by \eqref{b1-range}, \eqref{phi12-defn}, \eqref{Phi2-positive} and \eqref{a-ineqn}, 
\begin{align}
&0<\Phi_2(\widetilde{w},z)(s)\le C_3\int_{-\infty}^s\mbox{exp}\left(\beta\eta^{1-m}e^{-\frac{\rho_1}{\beta}s}(\rho-s)\right)\,d\rho\le C_4\eta^{m-1}e^{\frac{\rho_1}{\beta}s}\notag\\
\Rightarrow\quad& 0<\Phi_2(\widetilde{w},z)(s)e^{-\frac{\rho_1}{\beta}s}\le C_4\eta^{m-1}\quad\forall s<-b_1\label{phi2-upper-bd8}\\
\Rightarrow\quad& 0<\Phi_2(\widetilde{w},z)(s)e^{-\frac{\rho_1 s}{2\beta}}\le C_4\eta^{m-1}e^{-\frac{\rho_1b_1}{2\beta}}\quad\forall s<-b_1\notag\\
\Rightarrow\quad&\|\Phi_2(\widetilde{w},z)\|_{L^\infty\left((-\infty,-b_1);e^{-\frac{\rho_1}{2\beta}s}\right)}\le \3_1.\label{phi2-upper-bd9}
\end{align}
Now by \eqref{exp-term-bd} and \eqref{b1-range},
\begin{align}\label{phi1-eta-bd}
\left|\Phi_1(\widetilde{w},z)(s)-\eta\right|e^{-\frac{\rho_1s}{4\beta}}
=&\eta\left( \mbox{exp}\,\left(\int_{-\infty}^sz(\rho)\,d\rho\right)-1\right)e^{-\frac{\rho_1s}{4\beta}}\quad\forall s<-b_1\notag\\
\le&\eta\left(\mbox{exp}\,\left(\int_{-\infty}^sC_4\eta^{m-1}e^{\frac{\rho_1}{\beta}\rho}\,d\rho\right)-1\right)e^{-\frac{\rho_1s}{4\beta}}\quad\forall s<-b_1\notag\\
=&\eta\left(\mbox{exp}\,\left(\frac{C_3\eta^{m-1}}{\rho_1}e^{\frac{\rho_1s}{\beta}}\right)-1\right)e^{-\frac{\rho_1s}{4\beta}}\quad\forall s<-b_1\notag\\
\le&C_5\eta^me^{\frac{3\rho_1s}{4\beta}}\qquad\qquad\quad\forall s<-b_1\notag\\
\le&C_5\eta^me^{-\frac{3\rho_1b_1}{4\beta}}\qquad\qquad\,\forall s<-b_1\notag\\
\le&\3_1\qquad\qquad\qquad\qquad\,\forall s<-b_1.
\end{align}
By \eqref{Phi1-ge-eta}, \eqref{phi2-upper-bd8}, \eqref{phi2-upper-bd9} and \eqref{phi1-eta-bd}, we get $\Phi(\cD_{b_1})\subset \cD_{b_1}$. 

We will now show that when $b_1$ is sufficiently large, $\Phi :\cD_{b_1}\to \cD_{b_1}$ is a contraction map. Let $(\widetilde{w}_1,z_1), (\widetilde{w}_2,z_2)\in \cD_{b_1}$ and $\delta:=\|(\widetilde{w}_1,z_1)-(\widetilde{w}_2,z_2)\|_{\cX_{b_1}}$. Then 
\begin{equation}\label{tilde-wi-upper-lower-bd}
\eta\le\widetilde{w}_i(s)\le 3\eta/2\quad\forall s<b_1, i=1,2.
\end{equation}
Hence  by \eqref{tilde-wi-upper-lower-bd} and an argument similar to the proof of \eqref{a-ineqn},
\begin{equation}\label{a-wi-ineqn}
a(\widetilde{w}_i,\rho)-a(\widetilde{w}_i,s)
\le\beta\eta^{1-m}e^{-\frac{\rho_1}{\beta}s}(\rho-s)\quad\forall \rho<s<-b_1,i=1,2.
\end{equation}
By \eqref{tilde-wi-upper-lower-bd} and the mean value theorem for any $s<-b_1$ there exists a constant $\xi$ between $\int_{-\infty}^sz_1(\rho)\,d\rho$ and $\int_{-\infty}^sz_2(\rho)\,d\rho$ such that 
\begin{align}\label{phi1-12-l1}
|\Phi_1(\widetilde{w}_1,z_1)(s)-\Phi_1(\widetilde{w}_2,z_2)(s)|e^{-\frac{\rho_1s}{4\beta}}=&\eta\left|\mbox{exp}\,\left(\int_{-\infty}^sz_1(\rho)\,d\rho\right)-\mbox{exp}\,\left(\int_{-\infty}^sz_2(\rho)\,d\rho\right)\right|e^{-\frac{\rho_1s}{4\beta}}\notag\\
=&\eta e^{\xi}\left|\int_{-\infty}^s(z_1(\rho)-z_2(\rho))\,d\rho\right|e^{-\frac{\rho_1s}{4\beta}}\notag\\
\le&e^{-\frac{\rho_1s}{4\beta}}\max(\widetilde{w}_1(s),\widetilde{w}_2(s))\|z_1-z_2\|_{L^{\infty}\left((-\infty,-b_1);e^{-\frac{\rho_1}{2\beta}s}\right)}
\int_{-\infty}^se^{\frac{\rho_1}{2\beta}\rho}\,d\rho\notag\\
\le&\frac{3\beta\eta}{\rho_1}e^{-\frac{\rho_1b_1}{4\beta}}\delta\quad\forall s<-b_1.
\end{align}
By \eqref{phi12-defn},
\begin{align}\label{Phi2-difference-bd}
|\Phi_2(\widetilde{w}_1,z_1)(s)-\Phi_2(\widetilde{w}_2,z_2)(s)|
\le&C_3\int_{-\infty}^s\left|e^{a(\widetilde{w}_1,\rho)-a(\widetilde{w}_1,s)}-e^{a(\widetilde{w}_2,\rho)-a(\widetilde{w}_2,s)}\right|\,d\rho\notag\\
&\quad +m\int_{-\infty}^se^{a(\widetilde{w}_1,\rho)-a(\widetilde{w}_1,s)}\left|(z_1(\rho)+C_1)^2-(z_2(\rho)+C_1)^2\right|\,d\rho\notag\\
&\quad+m\int_{-\infty}^s\left|e^{a(\widetilde{w}_1,\rho)-a(\widetilde{w}_1,s)}-e^{a(\widetilde{w}_2,\rho)-a(\widetilde{w}_2,s)}\right|(z_2(\rho)+C_1)^2\,d\rho\notag\\
\le&(C_3+m(\3_1+|C_1|)^2)I_1+mI_2\quad\forall s<-b_1
\end{align}
where
\begin{equation*}
I_1=\int_{-\infty}^s\left|e^{a(\widetilde{w}_1,\rho)-a(\widetilde{w}_1,s)}-e^{a(\widetilde{w}_2,\rho)-a(\widetilde{w}_2,s)}\right|\,d\rho\quad\forall s<-b_1
\end{equation*}
and
\begin{equation*}
I_2=\int_{-\infty}^se^{a(\widetilde{w}_1,\rho)-a(\widetilde{w}_1,s)}|z_1(\rho)-z_2(\rho)|\left|z_1(\rho)+z_2(\rho)+2C_1\right|\,d\rho\quad\forall s<-b_1.
\end{equation*}
Now
\begin{align}\label{I1-integrand-ineqn}
\left|e^{a(\widetilde{w}_1,\rho)-a(\widetilde{w}_1,s)}-e^{a(\widetilde{w}_2,\rho)-a(\widetilde{w}_2,s)}\right|
=&\left|\int_0^1\frac{\1}{\1 t}e^{t(a(\widetilde{w}_1,\rho)-a(\widetilde{w}_1,s))+(1-t)(a(\widetilde{w}_2,\rho)-a(\widetilde{w}_2,s))}\,dt\right|\notag\\
\le &H(\rho)\int_0^1e^{t(a(\widetilde{w}_1,\rho)-a(\widetilde{w}_1,s))+(1-t)(a(\widetilde{w}_2,\rho)-a(\widetilde{w}_2,s))}\,dt\quad\forall \rho\le s<-b_1
\end{align}
where
\begin{align}\label{H-eqn}
H(\rho):=&\left|a(\widetilde{w}_1,\rho)-a(\widetilde{w}_1,s)-(a(\widetilde{w}_2,\rho)-a(\widetilde{w}_2,s))\right|\notag\\
=&\beta\left|\int_{\rho}^se^{-\frac{\rho_1}{\beta}\rho'}(\widetilde{w}_1^{1-m}(\rho')-\widetilde{w}_2^{1-m}(\rho'))\,d\rho'\right|\quad\forall \rho\le s<-b_1.
\end{align}
Now by \eqref{tilde-wi-upper-lower-bd},
\begin{align}\label{widetilde-w12-difference-bd}
\left|\widetilde{w}_1^{1-m}(\rho)-\widetilde{w}_2^{1-m}(\rho)\right|
=&\left|\int_0^1\frac{\1}{\1 s'}\left(s'\widetilde{w}_1(\rho)+(1-s')\widetilde{w}_2(\rho)\right)^{1-m}\,ds'\right|\notag\\
\le&(1-m)|\widetilde{w}_1(\rho)-\widetilde{w}_2(\rho)|\int_0^1\left(s'\widetilde{w}_1(\rho)+(1-s')\widetilde{w}_2(\rho)\right)^{-m}\,ds'\notag\\
\le&(1-m)\eta^{-m}|\widetilde{w}_1(\rho)-\widetilde{w}_2(\rho)|\quad\forall \rho<-b_1.
\end{align}
Hence by \eqref{H-eqn} and \eqref{widetilde-w12-difference-bd},
\begin{align}\label{H-ineqn}
H(\rho)\le&(1-m)\beta\eta^{-m}\|\widetilde{w}_1-\widetilde{w}_2\|_{L^\infty\left((-\infty,-b_1);e^{-\frac{\rho_1s}{4\beta}}\right)}\int_{\rho}^se^{-\frac{3\rho_1}{4\beta}\rho'}\,d\rho'\notag\\
\le&\frac{2\beta^2}{\rho_1\eta^m}\left(e^{-\frac{3\rho_1}{4\beta}\rho}-e^{-\frac{3\rho_1}{4\beta}s}\right)\delta
\qquad\qquad\qquad\forall \rho\le s<-b_1.
\end{align}
By \eqref{a-wi-ineqn}, \eqref{I1-integrand-ineqn} and \eqref{H-ineqn},
\begin{align}\label{exp-a-difference}
&\left|e^{a(\widetilde{w}_1,\rho)-a(\widetilde{w}_1,s)}-e^{a(\widetilde{w}_2,\rho)-a(\widetilde{w}_2,s)}\right|\notag\\
\le&\frac{2\beta^2}{\rho_1\eta^m}\left\{\mbox{exp}\,\left(\beta\eta^{1-m}e^{-\frac{\rho_1}{\beta}s}(\rho-s)-\frac{3\rho_1\rho}{4\beta}\right)
-\mbox{exp}\,\left(\beta\eta^{1-m}e^{-\frac{\rho_1}{\beta}s}(\rho-s)-\frac{3\rho_1 s}{4\beta}\right)\right\}\delta\quad\forall \rho\le s<-b_1.
\end{align}
Hence by \eqref{b1-range} and \eqref{exp-a-difference},
\begin{align}\label{I1-ineqn}
I_1\le&\frac{2\beta^2\delta}{\rho_1\eta^m}\int_{-\infty}^s\left\{\mbox{exp}\,\left(\beta\eta^{1-m}e^{-\frac{\rho_1}{\beta}s}(\rho-s)-\frac{3\rho_1\rho}{4\beta}\right)
-\mbox{exp}\,\left(\beta\eta^{1-m}e^{-\frac{\rho_1}{\beta}s}(\rho-s)-\frac{3\rho_1 s}{4\beta}\right)\right\}\,d\rho\notag\\
\le&\frac{2\beta^2\delta}{\rho_1\eta^m}\left(\frac{e^{-\frac{3\rho_1s}{4\beta}}}{\beta\eta^{1-m}e^{-\frac{\rho_1}{\beta}s}-\frac{3\rho_1}{4\beta}}
-\frac{e^{-\frac{3\rho_1s}{4\beta}}}{\beta\eta^{1-m}e^{-\frac{\rho_1}{\beta}s}}\right)\quad\forall s<-b_1\notag\\
\le&\frac{2\eta^{-1}e^{\frac{5\rho_1s}{4\beta}}\delta}{\beta\eta^{1-m}
-\frac{3\rho_1}{4\beta}e^{\frac{\rho_1s}{\beta}}}\quad\forall s<-b_1.
\end{align}
Now by \eqref{a-wi-ineqn},
\begin{align}\label{I2-ineqn}
I_2\le&(\|z_1\|_{L^{\infty}(-\infty,-b_1)}+\|z_2\|_{L^{\infty}(-\infty,-b_1)}+2|C_1|)\delta\int_{-\infty}^s\mbox{ exp}\,\left(\beta\eta^{1-m}e^{-\frac{\rho_1}{\beta}s}(\rho-s)+\frac{\rho_1}{2\beta}\rho\right)\,d\rho\notag\\
\le&\frac{2(\3_1+|C_1|)e^{\frac{\rho_1s}{2\beta}}}{\beta\eta^{1-m}
e^{-\frac{\rho_1}{\beta}s}+\frac{\rho_1}{2\beta}}\delta\quad\forall s<-b_1.
\end{align}
Hence by \eqref{Phi2-difference-bd}, \eqref{I1-ineqn} and \eqref{I2-ineqn},
\begin{equation}\label{phi2-12-l1}
|\Phi_2(\widetilde{w}_1,z_1)(s)-\Phi_2(\widetilde{w}_2,z_2)(s)|e^{-\frac{\rho_1s}{2\beta}}
\le\left\{\frac{2(C_3+m(\3_1+|C_1|)^2)\eta^{-1}e^{-\frac{3\rho_1b_1}{4\beta}}}{\beta\eta^{1-m}
-\frac{3\rho_1}{4\beta}e^{-\frac{\rho_1b_1}{\beta}}}
+\frac{2m(\3_1+|C_1|)}{\frac{\rho_1}{2\beta}+\beta\eta^{1-m}
e^{\frac{\rho_1}{\beta}b_1}}\right\}\delta
\end{equation}
holds for any  $s<-b_1$.
By choosing $b_1$ sufficiently large in \eqref{phi1-12-l1} and \eqref{phi2-12-l1} we get
\begin{equation}\label{phi1-contraction}
\|\Phi_1(\widetilde{w}_1,z_1)-\Phi_1(\widetilde{w}_2,z_2)\|_{L^\infty\left((-\infty,-b_1);e^{-\frac{\rho_1}{4\beta}s}\right)}\le\frac{\delta}{5}.
\end{equation}
and
\begin{equation}\label{phi2-contraction}
\|\Phi_2(\widetilde{w}_1,z_1)-\Phi_2(\widetilde{w}_2,z_2)\|_{L^\infty\left((-\infty,-b_1);e^{-\frac{\rho_1}{2\beta}s}\right)}\le\frac{\delta}{5}.
\end{equation}
By \eqref{phi1-contraction} and \eqref{phi2-contraction},  for sufficiently large $b_1$ we have
\begin{equation}\label{phi-contraction}
\|\Phi(\widetilde{w}_1,z_1)-\Phi(\widetilde{w}_2,z_2)\|_{\cX_{b_1}}\le\frac{\delta}{5}.
\end{equation}
Hence we can choose a  sufficiently large  $b_1$ such that the map $\Phi:\cD_{b_1}\to \cD_{b_1}$ is a contraction map with Lipschitz constant less than $1/5$. Since $\cD_{b_1}$ is a complete metric space, by the contraction mapping theorem there exists a unique fixed point 
$(\widetilde{w},z)=\Phi(\widetilde{w},z)\in\cD_{b_1}$ that satisfies \eqref{z-representation-formula},   \eqref{tilde-w-formula} and 
\begin{equation}\label{z-growth-rate-bd11}
\widetilde{w}(s)\ge\eta,\quad 0\le z(s)e^{-\frac{\rho_1}{\beta}s}\le C_4\eta^{m-1}\quad\forall s<-b_1.
\end{equation} 
Since $z(s)=\Phi_2(\widetilde{w},z)(s)$, by \eqref{Phi2-positive}, $z(s)>0$ for any $s<-b_1$. This together with  
\eqref{z-growth-rate-bd11} implies \eqref{z-growth-rate-bd} holds for any $s<-b_1$.
Multiplying \eqref{z-representation-formula} by $e^{a(\widetilde{w},s)}$ and then differentiating the equation with respect to $s$ we get that $z$ satisfies \eqref{z-eqn} in $(-\infty,-b_1)$.

We will now prove uniqueness of solution.
Suppose that $z_3\in C^1(-\infty,-b_1)$ is another solution of \eqref{z-eqn} in $(-\infty,-b_1)$   with $\widetilde{w}$ being replaced by  
\begin{equation}\label{tilde-w3-formula} 
\widetilde{w}_3(s)=w_3(e^s)=\eta\, \mbox{exp}\left(\int_{-\infty}^sz_3(\rho)\,d\rho\right)\quad\forall s<-b_1
\end{equation} 
which satisfies 
\begin{equation}\label{z3-growth-rate-bd}
0<z_3(s)e^{-\frac{\rho_1}{\beta}s}\le C_4\eta^{m-1}
\end{equation}
in $(-\infty,-b_1)$. Since  $b_1>b_0$, by \eqref{b1-range}, \eqref{tilde-w3-formula} and \eqref{z3-growth-rate-bd},
\begin{equation}\label{z3-weight-infty-norm-bd}
\|z_3\|_{L^\infty\left((-\infty,-b_1);e^{-\frac{\rho_1}{2\beta}s}\right)}\le C_4\eta^{m-1}e^{-\frac{\rho_1b_1}{2\beta}}<\3_1.
\end{equation} 
and 
\begin{equation}\label{w3-lower-bd}
\widetilde{w}_3(s)\ge\eta\quad\forall s<-b_1.
\end{equation} 
By  \eqref{tilde-w3-formula}, \eqref{z3-growth-rate-bd} and an argument similar to the proof of \eqref{phi1-eta-bd},
\begin{equation}\label{tilde-w-eta-bd}
\|\widetilde{w}_3-\eta\|_{L^\infty\left((-\infty,-b_1);e^{-\frac{\rho_1}{4\beta}s}\right)}\le\3_1\quad\forall s<-b_1.
\end{equation}
Moreover by the discussion in the introduction of the section, $z_3$ satisfies \eqref{z-representation-formula}
with $a(\widetilde{w},s)$ given by \eqref{a-defn} being replaced by $a(\widetilde{w}_3,s)$.
This together with  \eqref{z3-growth-rate-bd}, \eqref{z3-weight-infty-norm-bd}, \eqref{w3-lower-bd}  and \eqref{tilde-w-eta-bd} implies that $(\widetilde{w}_3,z_3)\in\cD_{b_1}$ and $(\widetilde{w}_3,z_3)=\Phi (\widetilde{w}_3,z_3)$. Thus $(\widetilde{w}_3,z_3)$ is a fixed point of the mapping $\Phi:\cD_{b_1}\to \cD_{b_1}$. Hence by the uniqueness of the fixed point for the mapping $\Phi:\cD_{b_1}\to \cD_{b_1}$, $\widetilde{w}_3=\widetilde{w}$, $z_3=z$, in $(-\infty,-b_1)$ and the lemma follows.

\end{proof}

\begin{thm}\label{z-existence-thm}
Let $n\ge 3$, $0<m<\frac{n-2}{n}$, $\rho_1>0$, $\eta>0$, $\beta>\frac{m\rho_1}{n-2-nm}$ and $\alpha$ satisfies \eqref{alpha-beta-relation2}.   Then  the equation \eqref{z-eqn} has a unique  solution $z\in C^1(\mathbb{R})$   in $\mathbb{R}$ that satisfies \eqref{z-growth-rate-bd} in $\mathbb{R}$
with $\widetilde{w}$ given by  \eqref{tilde-w-formula} where $C_4$ is given by \eqref{c4-defn}. 
\end{thm} 
\begin{proof}
By Lemma \ref{z-local-existence-lem} there exists a constant $b_1>0$ such that the equation \eqref{z-eqn}
has a unique  solution $z\in C^1(-\infty, b_1)$   in $(-\infty,-b_1)$ that satisfies 
\eqref{z-growth-rate-bd} in $(-\infty, -b_1)$ with $\widetilde{w}$ given by 
\eqref{tilde-w-formula} in $(-\infty, -b_1)$. 
Let $(-\infty,b_2)$ be the maximal interval of existence of solution $z$
of \eqref{z-eqn} in $(-\infty,b_2)$ that satisfies \eqref{z-growth-rate-bd} in $(-\infty, b_2)$ with  $\widetilde{w}$ given by \eqref{tilde-w-formula} in $(-\infty, b_2)$. Then $b_2\ge -b_1$.

Suppose $b_2<\infty$. We choose a sequence $\{s_i\}_{i=1}^{\infty}\in (-\infty, b_2)$ such that $s_i\to b_2$ as $i\to\infty$. Then by \eqref{z-growth-rate-bd} and passing to a subsequence if necessary $\lim_{i\to\infty}z(s_i)$ exists  and
\begin{equation}\label{z-si-limit-bd}
0\le \lim_{i\to\infty}z(s_i)\le C_4e^{\frac{\rho_1b_2}{\beta}}.
\end{equation}
Then by \eqref{z-eqn}, \eqref{tilde-w-formula} and \eqref{z-si-limit-bd} we get that $z$ and $\widetilde{w}$ satisfies an ODE system with uniformly Lipschitz coefficients in neighbourhoods of $(z(s_i),\widetilde{w}(s_i))$, $i\in\Z^+$. Hence by standard ODE theory there exists a constant $\delta_1>0$ such that for any $i\in\mathbb{Z}^+$ the equation \eqref{z-eqn} has a solution $z_1$ in $(s_i, s_i+\delta_1)$ with $z_1(s_i)=z(s_i)$  and $\widetilde{w}(s)$ being replaced by
\begin{equation*}
\widetilde{w}_1(s)=\widetilde{w}(s_i)\, \mbox{exp}\left(\int_{s_i}^sz_1(\rho)\,d\rho\right)\quad\forall s_i\le s<s_i+\delta_1.
\end{equation*}
We now choose $i_0\in\mathbb{Z}^+$ such that $s_{i_0}+\delta_1>b_2$. We extend $z$, $w$, $\widetilde{w}$, to functions on $(-\infty,s_{i_0}+\delta_1)$ by defining $z(s)=z_1(s)$ and $\widetilde{w}(s)=\widetilde{w}_1(s)$ for any
$s\in (s_{i_0},s_{i_0}+\delta_1)$. Then $z$ satisfies \eqref{z-eqn} in $(-\infty, s_{i_0}+\delta_1)$ with $\widetilde{w}$
satisfying  \eqref{tilde-w-formula} in $(-\infty, s_{i_0}+\delta_1)$. We claim that 
\begin{equation}\label{z-positive3}
z(s)>0\quad\forall s<s_{i_0}+\delta_1.
\end{equation}
Suppose the claim does not hold. Then there exists a constant $b_3\in [-b_1,s_{i_0}+\delta_1)$ such that $z(b_3)=0$. Let
\begin{equation*}
b_4=\sup\{b\in\mathbb{R}:z(s)>0\quad\forall s<b\}.
\end{equation*}
Then since \eqref{z-growth-rate-bd} holds for any $s<-b_1$, we have $-b_1\le b_4\le b_3$, $z(s)>0$ for any $s<b_4$ and $z(b_4)=0$. Hence $z_s(b_4)\le 0$. On the other hand by putting $s=b_4$ in \eqref{z-eqn} we have
\begin{equation*}
z_s(b_4)=\frac{\alpha}{\beta}\left(n-2-\frac{m\alpha}{\beta}\right)
=\frac{\alpha}{\beta^2(1-m)(n-2-nm)}\left(\beta-\frac{m\rho_1}{n-2-nm}\right)>0
\end{equation*}
and contradiction arises. Hence no such constant $b_3$ exists and \eqref{z-positive3} holds. 
By \eqref{z-eqn}, $z$ satisfies 
\eqref{z-representation-formula} in $(-\infty, s_{i_0}+\delta_1)$. Then by \eqref{z-positive3},
\begin{equation}\label{tilde-w-lower-bd5}
\widetilde{w}(s)\ge\eta\quad\forall s<s_{i_0}+\delta_1.
\end{equation}
Then by \eqref{z-representation-formula}, \eqref{z-positive3}, \eqref{tilde-w-lower-bd5} and an argument similar as before, $z$ satisfies \eqref{z-growth-rate-bd} for any $s<s_{i_0}+\delta_1$. This contradicts the maximality of the choice of $b_2$. Hence $b_2=\infty$ and $z$ is a solution of  \eqref{z-eqn} in $\mathbb{R}$ that satisfies \eqref{z-growth-rate-bd} in $\mathbb{R}$
with $\widetilde{w}$ given by  \eqref{tilde-w-formula}.

Suppose now $z_3\in C^1(-\infty, b_1)$ is another solution of \eqref{z-eqn} in $\mathbb{R}$ that satisfies \eqref{z3-growth-rate-bd}  in $\mathbb{R}$ with $\widetilde{w}(s)$ being replaced by $\widetilde{w}_3$ given by \eqref{tilde-w3-formula}. Then  by \eqref{z-eqn} for $z_3$ and the argument in the beginning of this section,  $z_3$ satisfies 
\begin{equation}\label{z3-representation-formula}
z_3(s)=C_3\int_{-\infty}^se^{a(\widetilde{w}_3,\rho)-a(\widetilde{w}_3,s)}\,d\rho -m\int_{-\infty}^se^{a(\widetilde{w}_3,\rho)-a(\widetilde{w}_3,s)}(z_3(\rho)+C_1))^2\,d\rho
\end{equation}
in $\mathbb{R}$.
 Let $\3_1>0$, $b_1>b_0>0$ and the map $\Phi:\cD_{b_1}\to \cD_{b_1}$ be as in the proof of Lemma  \ref{z-local-existence-lem}. 
Then by \eqref{tilde-w3-formula}, \eqref{z3-growth-rate-bd}  and  an argument similar to the proof of Lemma  \ref{z-local-existence-lem}, we get \eqref{z3-weight-infty-norm-bd}, \eqref{w3-lower-bd} and
\eqref{tilde-w-eta-bd}. These together with  \eqref{z3-representation-formula} implies that $(\widetilde{w}_3,z_3)\in\cD_{b_1}$ and $(\widetilde{w}_3,z_3)=\Phi (\widetilde{w}_3,z_3)$. Thus $(\widetilde{w}_3,z_3)$ is a fixed point of the mapping $\Phi:\cD_{b_1}\to \cD_{b_1}$. Hence by the uniqueness of the fixed point for the mapping $\Phi:\cD_{b_1}\to \cD_{b_1}$, $\widetilde{w}_3=\widetilde{w}$, $z_3=z$, in $(-\infty,-b_1)$. We now choose $b_2<-b_1$. Then both $z$ and $z_3$ satisfies \eqref{z-eqn} in $(b_2,\infty)$ with $\widetilde{w}=\widetilde{w}, \widetilde{w}_3$, respectively and $z(b_2)=z_3(b_2)$. Hence by the standard ODE theory, $z=z_3$ in $[b_2,\infty)$. Thus $z=z_3$ in $\mathbb{R}$ and the theorem follows.

\end{proof}

\noindent{\ni{\it Proof of Theorem \ref{blow-up-self-similar-soln-thm}:}} We divide the proof of Theorem \ref{blow-up-self-similar-soln-thm} into two cases.

\noindent $\underline{\text{\bf Case 1}}$: In this case we will prove that there exists a  unique solution $f=f^{(m)}$ of \eqref{f-ode} in $C^2(0,\infty)$ that satisfies \eqref{blow-up-rate-x=0} and
\eqref{z-growth-rate-bd} in $\mathbb{R}$ where $C_4$ is given by \eqref{c4-defn}.
 
\noindent $\text{\it Proof of case 1}$: By Theorem \ref{z-existence-thm}, the equation \eqref{z-eqn}
has a unique  solution $z\in C^1(\mathbb{R})$   in $\mathbb{R}$ that satisfies 
\eqref{z-growth-rate-bd} with $\widetilde{w}$ given by \eqref{tilde-w-formula}  in $\mathbb{R}$. Let $w$ be given by \eqref{tilde-w-defn}. By \eqref{z-growth-rate-bd} and \eqref{tilde-w-formula},
\begin{equation}\label{w-limit-at-origin}
\lim_{r\to 0^+}w(r)=\eta.
\end{equation} 
Hence by defining $w(0)=\eta$ we extend $w$ to a continuous function on $\mathbb{R}$.
Differentiating \eqref{tilde-w-formula} with respect to $s$,
\begin{equation}\label{z-w-wr-relation}
z(s)=\frac{\widetilde{w}_s(s)}{\widetilde{w}(s)}=\frac{rw_r(r)}{w(r)}\quad\forall 0<r=e^s\in\mathbb{R}.
\end{equation}
Substituting \eqref{z-w-wr-relation} into \eqref{z-eqn}, we get \eqref{w1-r-eqn}. 
Let
\begin{equation*}
f(r)=r^{-\alpha/\beta}w(r).
\end{equation*}
Substituting $w(r)=r^{\alpha/\beta}f(r)$ into \eqref{w1-r-eqn}, we get that $f$ satisfies 
\eqref{f-ode}.
 By \eqref{w-limit-at-origin}, $f$ satisfies \eqref{blow-up-rate-x=0}.

We will now prove the uniqueness of solution $f$.
Suppose $f_3\in C^2(0,\infty)$ is another solution of \eqref{f-ode}, \eqref{blow-up-rate-x=0}, that satisfies \eqref{z3-growth-rate-bd} in $\mathbb{R}$  where  $s=\log r$, 
\begin{equation}\label{w3-defn2}
w_3(r)=r^{\alpha/\beta}f_3(r),\quad\widetilde{w}_3(s)=w_3(e^s)\quad\forall r>0, s\in\mathbb{R}
\end{equation}
and 
\begin{equation}\label{z3-w3-w3r-relation}
z_3(s)=\frac{\widetilde{w}_{3,s}(s)}{\widetilde{w}_3(s)}=\frac{rw_{3,r}(r)}{w_3(r)}\quad\forall 0<r=e^s\in\mathbb{R}.
\end{equation} 
Then by the discussion in the beginning of this section, $z_3\in C^1(\mathbb{R})$ satisfies \eqref{z-eqn} with 
$\widetilde{w}$ being replaced by $\widetilde{w}_3$. Then by Theorem \ref{z-existence-thm}, $z=z_3$ and $w=w_3$. Thus $f=f_3$ and case 1 follows.

\noindent $\underline{\text{\bf Case 2}}$: In this case we will prove that there exists a  unique solution $f=f^{(m)}\in C^2(0,\infty)$ of \eqref{f-ode} that satisfies \eqref{blow-up-rate-x=0}.
 
\noindent $\text{\it Proof of case 2}$: Existence of  solution $f=f^{(m)}\in C^2(0,\infty)$ of  \eqref{f-ode} that satisfies \eqref{blow-up-rate-x=0} and \eqref{z-growth-rate-bd} in $\mathbb{R}$ is given by case 1. 
Hence we only need to prove uniqueness of solution $f\in C^2(0,\infty)$ of  \eqref{f-ode} that satisfies \eqref{blow-up-rate-x=0}  without the requirement that \eqref{z-growth-rate-bd} holds in $\mathbb{R}$.

Suppose $f_3\in C^2(0,\infty)$ is another solution of \eqref{f-ode} and \eqref{blow-up-rate-x=0}. 
Let $s=\log r$, $w_3(r)$, $\widetilde{w}_3(s)$ and $z_3(s)$ be given by \eqref{w3-defn2} and \eqref{z3-w3-w3r-relation}. We claim that $z_3$ also satisfies \eqref{z3-growth-rate-bd} in $\mathbb{R}$. In order to prove this claim we first observe that by  \eqref{z3-w3-w3r-relation} and the discussion in the beginning of section 2, $z_3\in C^1(\mathbb{R})$ satisfies 
\begin{align}
&z_{3,s}+\left(n-2-\frac{2m\alpha}{\beta}\right)z_3+mz_3^2+\beta e^{-\frac{\rho_1}{\beta}s}\widetilde{w}_3^{1-m}z_3=\frac{\alpha}{\beta}\left(n-2-\frac{m\alpha}{\beta}\right)=:C_2\quad\forall s\in\mathbb{R}\label{z3-eqn}\\
\Leftrightarrow\quad&z_{3,s}+m(z_3+C_1)^2+\beta e^{-\frac{\rho_1}{\beta}s}\widetilde{w}_3^{1-m}z_3=C_3\quad\forall s\in\mathbb{R}\label{z3-eqn2}
\end{align} 
where $C_1$, $C_3$, are given by \eqref{c123-defn} and $C_2>0$.
By \eqref{blow-up-rate-x=0} for $f_3$ and \eqref{w3-defn2}, 
\begin{equation}\label{w3-limit3}
\lim_{r\to 0}w_3(r)=\eta.
\end{equation}
 Hence by 
letting $w_3(0)=\eta$, we extend $w_3$ to a continuous function in $[0,\infty)$. Then there exists a constant $b_1>0$ such that
\begin{equation}\label{tilde-w3-upper-lower-bd}
\eta/2\le \widetilde{w}_3(s)\le 3\eta/2\quad\forall s<-b_1.
\end{equation}
We now choose a sequence $\{s_i\}_{i=1}^{\infty}\subset (-\infty, -b_1)$ such that 
\begin{equation}\label{si-limit}
s_{i+1}<s_i-i\quad\forall i\in\mathbb{Z}^+.
\end{equation}
Then by \eqref{w3-limit3}, \eqref{tilde-w3-upper-lower-bd}, \eqref{si-limit} and the mean value theorem for any $i\in\mathbb{Z}^+$, there exists a constant $\xi_i\in (s_{i+1},s_i)$ such that
\begin{align}\label{z3-limit}
&|\widetilde{w}_{3,s}(\xi_i)|=\left|\frac{\widetilde{w}_3(s_i)-\widetilde{w}_3(s_{i+1})}{s_i-s_{i+1}}\right|
\le\frac{\eta}{i}\quad\forall i\in\mathbb{Z}^+\notag\\
\Rightarrow\quad&\lim_{i\to\infty}\widetilde{w}_{3,s}(\xi_i)=0\notag\\
\Rightarrow\quad&\lim_{i\to\infty}z_3(\xi_i)=\lim_{i\to\infty}\frac{\widetilde{w}_{3,s}(\xi_i)}{\widetilde{w}_3(\xi_i)}=\frac{0}{\eta}=0.
\end{align}
For any $j\in\mathbb{Z}^+$, let $a_j=\xi_j+(2/C_2)(1+|z_3(\xi_j)|)$ and let $M=\sup_{j\in\mathbb{Z}^+}|z_3(\xi_j)|$. By \eqref{z3-limit}, $M<\infty$.  Choose $b_2>b_1$ such that
\begin{equation*}
n-2-\frac{2m\alpha}{\beta}-mM+\beta e^{-\frac{\rho_1}{\beta}s}(\eta/2)^{1-m}>0\quad\forall s<-b_2.
\end{equation*}
Then
\begin{equation}\label{ineqn1}
n-2-\frac{2m\alpha}{\beta}+mz(\xi_j)+\beta e^{-\frac{\rho_1}{\beta}s}(\eta/2)^{1-m}>0\quad\forall s<-b_2, j\in\mathbb{Z}^+.
\end{equation}
By choosing a subsequence if necessary we may assume without loss of generality  that $a_j<-b_2$  for any $j\in\mathbb{Z}^+$. 
We claim that for any $i\in\mathbb{Z}^+$,
\begin{equation}\label{z3-positive-local}
\mbox{ there exists }a_i'\in (\xi_i,a_i)\mbox{ such that }z_3(a_i')>0. 
\end{equation} 
For any $i\in\mathbb{Z}^+$ if $z_3(\xi_i)>0$ , then \eqref{z3-positive-local} follows immediately.  Hence if the claim \eqref{z3-positive-local} does not hold, then there exists $i\in\mathbb{Z}^+$ such that
\begin{equation}\label{z-xi-negative}
z_3(\xi_i)\le 0.
\end{equation}
Then by \eqref{z3-eqn}, \eqref{tilde-w3-upper-lower-bd}, \eqref{ineqn1} and \eqref{z-xi-negative} we have
\begin{align}\label{z3s-positive}
C_2=&z_{3,s}(\xi_i)+\left(n-2-\frac{2m\alpha}{\beta}+mz_3(\xi_i)+\beta e^{-\frac{\rho_1}{\beta}\xi_i}\widetilde{w}_3^{1-m}(\xi_i)\right)z_3(\xi_i)\notag\\
\le &z_{3,s}(\xi_i)+\left(n-2-\frac{2m\alpha}{\beta}+mz_3(\xi_i)+\beta e^{-\frac{\rho_1}{\beta}\xi_i}(\eta/2)^{1-m}\right)z_3(\xi_i)\notag\\
\le&z_{3,s}(\xi_i).
\end{align}
By \eqref{z3s-positive} and continuity of $z_{3,s}$,  there exists a constant $\delta_i>0$ such that
\begin{align}\label{z3-increase}
&z_{3,s}(\xi)\ge C_2/2>0\quad\forall\xi_i\le\xi<\xi_i+\delta_i\notag\\
\Rightarrow\quad&z_3(\xi_i)<z_3(\xi)\quad\forall \xi_i<\xi<\xi_i+\delta_i.
\end{align}

If $z_3(\xi_i)=0$, \eqref{z3-positive-local} follows immediately from \eqref{z3-increase}. Hence we may assume without loss of generality that $z_3(\xi_i)<0$. We then divide the proof of the claim \eqref{z3-positive-local} into three cases.

\noindent $\underline{\text{\bf Case 2a}}$: There exists $\xi_i''\in (\xi_i,a_i)$ such that $z_3(\xi_i)\ge z_1(\xi_i'')$.

\noindent Let $a_i''=\sup\{a_0>\xi_i:z_3(\xi)>z_3(\xi_i)\quad\forall \xi_i<\xi<a_0\}$.
By \eqref{z3-increase}, $\xi_i+\delta_i\le a_i''\le \xi_i''$ and
\begin{equation}\label{z3-identity}
z_3(\xi)>z_3(\xi_i)\quad\forall\xi_i<\xi<a_i''\quad\mbox{ and }\quad z_3(a_i'')=z_3(\xi_i).
\end{equation}
Hence $z_{3,s}(a_i'')\le 0$. On the other hand by \eqref{z3-eqn}, \eqref{tilde-w3-upper-lower-bd}, \eqref{ineqn1}, \eqref{z-xi-negative} and \eqref{z3-identity} we have
\begin{align*}
C_2=&z_{3,s}(a_i'')+\left(n-2-\frac{2m\alpha}{\beta}+mz_3(a_i'')+\beta e^{-\frac{\rho_1}{\beta}a_i''}\widetilde{w}_3^{1-m}(a_i'')\right)z_3(a_i'')\notag\\
\le &z_{3,s}(a_i'')+\left(n-2-\frac{2m\alpha}{\beta}+mz_3(\xi_i)+\beta e^{-\frac{\rho_1}{\beta}a_i''}(\eta/2)^{1-m}\right)z_3(\xi_i)\notag\\
\le&z_{3,s}(a_i'')
\end{align*}
and contradiction arises. Hence no such $\xi_i''$ exists.

\noindent $\underline{\text{\bf Case 2b}}$: $z_3(\xi_i)<z_3(\xi)\le 0\quad\forall \xi_i<\xi<a_i$.

\noindent Then by \eqref{z3-eqn}, \eqref{tilde-w3-upper-lower-bd} and \eqref{ineqn1} we have
\begin{align}\label{z3s-positive2}
C_2=&z_{3,s}(\xi)+\left(n-2-\frac{2m\alpha}{\beta}+mz_3(\xi)+\beta e^{-\frac{\rho_1}{\beta}\xi}\widetilde{w}_1^{1-m}(\xi)\right)z_3(\xi)\notag\\
\le &z_{3,s}(\xi)+\left(n-2-\frac{2m\alpha}{\beta}+mz_3(\xi_i)+\beta e^{-\frac{\rho_1}{\beta}\xi}(\eta/2)^{1-m}\right)z_3(\xi)\notag\\
\le&z_{3,s}(\xi)\quad \forall \xi_i<\xi<a_i.
\end{align}
By \eqref{z3s-positive2},
\begin{equation}\label{z3-z3-positive}
z_3(a_i)>z_3(\xi_i)+C_2(a_i-\xi_i)=z_3(\xi_i)+2(1+|z_3(\xi_i)|)=2+|z_3(\xi_i)|>0.
\end{equation}
By \eqref{z3-z3-positive} and the intermediate value theorem, there exists $a_i'\in (\xi_i,a_i)$ such that \eqref{z3-positive-local} holds. 

\noindent $\underline{\text{\bf Case 2c}}$: $z_3(\xi_i)<z_3(\xi)\quad\forall \xi_i<\xi<a_i$ and there exists $a_i'\in (\xi_i,a_i)$ such that $z_3(a_i')>0$.

\noindent In this case \eqref{z3-positive-local} holds.

Hence the claim \eqref{z3-positive-local} holds. 
By \eqref{z3-positive-local} and an argument similar to the proof of \eqref{z-positive3} in the proof of case 1 we have
\begin{equation}\label{z3-positive10}
z_3(s)>0\quad\forall s>a_j',j\in\mathbb{Z}^+.
\end{equation}
Note that
\begin{equation*}
0<a_j-\xi_j\le\frac{2}{C_2}\left(1+\max_{j\in\mathbb{Z}^+}|z_1(\xi_j)|\right)<\infty\quad\forall j\in\mathbb{Z}^+.
\end{equation*}
Then we have
\begin{equation}\label{aj'-limit}
\lim_{j\to\infty}a_j'=\lim_{j\to\infty}a_j=\lim_{j\to\infty}\xi_j=-\infty.
\end{equation}
Hence by \eqref{z3-positive10} and \eqref{aj'-limit} we have
\begin{equation}\label{z3-positive11}
z_3(s)>0\quad\forall s\in\mathbb{R}.
\end{equation}
Then by \eqref{z3-w3-w3r-relation}, \eqref{w3-limit3} and \eqref{z3-positive11} we have
\begin{equation}\label{tilde-w3-lower-bd}
\widetilde{w}_3(s)\ge\eta\quad\forall s\in\mathbb{R}.
\end{equation}
Let
\begin{equation}\label{a-defn2}
a(\widetilde{w}_3,s)=-\beta\int_s^0e^{-\frac{\rho_1}{\beta}\rho'}\widetilde{w}_3^{1-m}(\rho')\,d\rho'\quad\forall s\in\mathbb{R}.
\end{equation}
Then by \eqref{tilde-w3-lower-bd} and an argument similar to the proof of \eqref{a-ineqn} we have
\begin{equation}\label{a-ineqn2}
a(\widetilde{w}_3,\rho)-a(\widetilde{w}_3,s)\le\beta \eta^{1-m}e^{-\frac{\rho_1}{\beta}s}(\rho-s)\quad\forall \rho<s.
\end{equation}
Multiplying \eqref{z3-eqn2} by $e^{a(\widetilde{w}_3,s)}$ and integrating over $(\xi_i,s)$, we get
\begin{equation*}
e^{a(\widetilde{w}_3,s)}z_3(s)=e^{a(\widetilde{w}_3,\xi_i)}z_3(\xi_i)+C_3\int_{\xi_i}^se^{a(\widetilde{w}_3,\rho)}\,d\rho -m\int_{\xi_i}^se^{a(\widetilde{w}_3,\rho)}(z_3(\rho)+C_1))^2\,d\rho\quad\forall s\in\mathbb{R}.
\end{equation*}
Hence $z_3$ satisfies
\begin{equation}\label{z3-representation-formula0}
z_3(s)=e^{a(\widetilde{w}_3,\xi_i)-a(\widetilde{w}_3,s)}z_3(\xi_i)+C_3\int_{\xi_i}^se^{a(\widetilde{w}_3,\rho)-a(\widetilde{w}_3,s)}\,d\rho -m\int_{\xi_i}^se^{a(\widetilde{w}_3,\rho)-a(\widetilde{w}_3,s)}(z_3(\rho)+C_1))^2\,d\rho\quad\forall s\in\mathbb{R}.
\end{equation}
Letting $i\to\infty$ in \eqref{z3-representation-formula0}, by \eqref{z3-limit} and \eqref{a-ineqn2},
\begin{align}\label{z3-representation-formula2}
z_3(s)=&C_3\int_{-\infty}^se^{a(\widetilde{w}_3,\rho)-a(\widetilde{w}_3,s)}\,d\rho -m\int_{-\infty}^se^{a(\widetilde{w}_3,\rho)-a(\widetilde{w}_3,s)}(z_3(\rho)+C_1))^2\,d\rho\quad\forall s\in\mathbb{R}
\notag\\
\le&C_3\int_{-\infty}^se^{a(\widetilde{w}_3,\rho)-a(\widetilde{w}_3,s)}\,d\rho\quad\forall s\in\mathbb{R}.
\end{align}
By \eqref{z3-positive11}, \eqref{a-ineqn2} and \eqref{z3-representation-formula2} and an argument similar to the proof of \eqref{phi2-upper-bd8} we get that $z_3$ satisfies \eqref{z3-growth-rate-bd}
in $\mathbb{R}$. Hence by Theorem \ref{z-existence-thm}, $z=z_3$ and $w=w_3$. Thus $f=f_3$ and the theorem follows.

{\hfill$\square$\vspace{6pt}}

By the proof of Theorem \ref{blow-up-self-similar-soln-thm} we have the following result.

\begin{cor}\label{blow-up-self-similar-soln-decay-rate-cor}
Let $n\ge 3$, $0<m<\frac{n-2}{n}$, $\rho_1>0$, $\eta>0$, $\beta>\frac{m\rho_1}{n-2-nm}$ and $\alpha$ satisfies \eqref{alpha-beta-relation2}. Let $f\in C^2(0,\infty)$ be the  unique   solution of \eqref{f-ode}  given by Theorem \ref{blow-up-self-similar-soln-thm} which satisfies \eqref{blow-up-rate-x=0}  where $s=\log r$, and $w$, $z$,  $\widetilde{w}$, are given by \eqref{w-defn}, \eqref{z-defn} and \eqref{tilde-w-defn}. Let $z$ be given by \eqref{z-defn}. Then $z$ satisifies \eqref{z-growth-rate-bd} in $\mathbb{R}$.
\end{cor}

\begin{cor}\label{f-upper-lower-bds-cor}(cf. Lemma 2.1 of \cite{HuK3})
Let $n\ge 3$, $0<m<\frac{n-2}{n}$, $\rho_1>0$, $\eta>0$, $\beta>\frac{m\rho_1}{n-2-nm}$ and $\alpha$  satisfies \eqref{alpha-beta-relation2}. 
Let $f\in C^2(0,\infty)$ be the  unique   solution of \eqref{f-ode}  given by Theorem \ref{blow-up-self-similar-soln-thm} which satisfies \eqref{blow-up-rate-x=0} where $s=\log r$, and $w$, $z$,  $\widetilde{w}$, are given by \eqref{w-defn}, \eqref{z-defn} and \eqref{tilde-w-defn}. Then
\begin{equation}\label{f-fr-ineqn}
f+\frac{\beta}{\alpha}rf_r>0\quad\forall r>0
\end{equation}
and
\begin{equation}\label{f-upper-lower-bds}
\eta r^{-\alpha/\beta}< f(r)\le \eta r^{-\alpha/\beta}\mbox{exp}\,\left(\frac{C_3\eta^{m-1}}{\rho_1}r^{\rho_1/\beta}\right)\quad\forall r>0.
\end{equation}
where $C_3>0$ is given by \eqref{c123-defn}. 
\end{cor} 
\begin{proof}
\eqref{f-fr-ineqn} and the left hand side of \eqref{f-upper-lower-bds} is proved in Lemma 2.1 of \cite{HuK3}. For the sake of completeness we will give a simple proof of \eqref{f-fr-ineqn} and  \eqref{f-upper-lower-bds} here. By Corollary \ref{blow-up-self-similar-soln-decay-rate-cor} \eqref{z-growth-rate-bd} holds in $\mathbb{R}$.
Substituting \eqref{w-defn} into \eqref{z-defn} and using \eqref{z-growth-rate-bd}, we get \eqref{f-fr-ineqn}. By \eqref{blow-up-rate-x=0}, \eqref{z-growth-rate-bd}, \eqref{c4-defn} and \eqref{z-w-wr-relation},  
\begin{align*}
&0<(\log w)_r\le C_3\beta^{-1}\eta^{m-1}r^{\frac{\rho_1}{\beta}-1}\quad\forall r>0\notag\\
\Rightarrow\quad&\eta<r^{\alpha/\beta}f(r)=w(r)\le\eta \,\mbox{exp}\,\left(\frac{C_3\eta^{m-1}}{\rho_1}r^{\rho_1/\beta}\right)\quad\forall r>0
\end{align*}
and \eqref{f-upper-lower-bds} follows.
\end{proof}

\begin{cor}\label{fr-negative2-cor}(Corollary 2.4 of \cite{HuK3})
Let $n\ge 3$, $0<m<\frac{n-2}{n}$, $\rho_1>0$, $\eta>0$, $\beta>\frac{m\rho_1}{n-2-nm}$ and $\alpha$  satisfies \eqref{alpha-beta-relation2}. 
Let $f\in C^2(0,\infty)$ be the  unique   solution of \eqref{f-ode} given by Theorem \ref{blow-up-self-similar-soln-thm} which satisfies \eqref{blow-up-rate-x=0}  where $s=\log r$, and $w$, $z$,  $\widetilde{w}$, are given by \eqref{w-defn}, \eqref{z-defn} and \eqref{tilde-w-defn}. Then
\begin{equation}\label{fr-negative2}
f_r<0\quad\forall r>0.
\end{equation}
\end{cor}
\begin{proof}
\eqref{fr-negative2} is prove in Corollary 2.4 of \cite{HuK3} by using the asymptotic expansion of $f$ near the origin. Here we will give a simple different proof of \eqref{fr-negative2}.
By \eqref{elliptic-eqn} and \eqref{f-fr-ineqn},
\begin{equation}\label{frr-ineqn}
(r^{n-1}(f^m)_r)_r=-(\alpha f+\beta rf_r)<0\quad\forall r>0.
\end{equation}
By the proof of Theorem \ref{blow-up-self-similar-soln-thm} there exists a sequence $\{\xi_i\}_{I=1}^{\infty}$, $\xi_i\to -\infty$ as $i\to\infty$, such that 
\begin{equation}\label{z-limit-at-neg-infty}
\lim_{i\to\infty}z(\xi_i)=0.
\end{equation} 
Let $r_i=e^{\xi_i}$ for any $i\in\mathbb{Z}^+$. By \eqref{w-defn}, \eqref{z-defn} and \eqref{z-limit-at-neg-infty},
\begin{align}\label{rfr-negative-near-0}
&\lim_{i\to\infty}\left(\frac{\alpha}{\beta}+\frac{r_if_r(r_i)}{f(r_i)}\right)=\lim_{i\to\infty}\frac{r_iw_r(r_i)}{w(r_i)}=\lim_{i\to\infty}z(\xi_i)=0\notag\\
\Rightarrow\quad&\lim_{i\to\infty}\frac{r_if_r(r_i)}{f(r_i)}=-\frac{\alpha}{\beta}<0.
\end{align}
By \eqref{blow-up-rate-x=0} and \eqref{rfr-negative-near-0} there exists a constant $i_0\in\mathbb{Z}^+$ such that
\begin{equation}\label{rfr-negative2}
f_r(r_i)\le -\frac{\alpha}{2\beta}<0\quad\forall i\ge i_0.
\end{equation}
Integrating \eqref{frr-ineqn} over $(r_i,r)$, by \eqref{rfr-negative2}, 
\begin{align*}
&mr^{n-1}f^{m-1}(r)f_r(r)<mr_i^{n-1}f^{m-1}(r_i)f_r(r_i)<0\quad\forall r>r_i, i\ge i_0\\
\Rightarrow\quad&f_r(r)<0\quad\forall r>0\quad\mbox{ as }i\to\infty
\end{align*}
and \eqref{fr-negative2} follows.

\end{proof}

\begin{prop}\label{f-upper-bd-lem}(cf. Lemma 4.1 of \cite{HuK3})
Let $n\ge 3$, $0<m<\frac{n-2}{n+2}$, $\rho_1>0$, $\eta>0$, 
\begin{equation}\label{beta-lower-bd3}
\beta>\max\left(\frac{m\rho_1}{n-2-nm},\frac{2m\rho_1}{n-2-(n+2)m}\right)
\end{equation}
and $\alpha$  satisfies \eqref{alpha-beta-relation2}. Let $C_2$ be given by \eqref{c123-defn}.
Let $f\in C^2(0,\infty)$ be the  unique  solution of \eqref{f-ode} given by Theorem \ref{blow-up-self-similar-soln-thm} which satisfies \eqref{blow-up-rate-x=0}  where $s=\log r$, and $w$, $z$,  $\widetilde{w}$, are given by \eqref{w-defn}, \eqref{z-defn} and \eqref{tilde-w-defn}. Then
\begin{equation}\label{z-upper-bd2}
0<z(s)\le\frac{C_2\eta^{m-1}}{\beta}e^{\frac{\rho_1s}{\beta}}\quad\forall s\in\mathbb{R}
\end{equation}
and
\begin{equation}\label{f-upper-bd11}
f(r)\le \eta r^{-\alpha/\beta}\mbox{exp}\,\left(\frac{C_2\eta^{m-1}}{\rho_1}r^{\rho_1/\beta}\right)\quad\forall r>0.
\end{equation}
\end{prop}
\begin{proof}
\eqref{f-upper-bd11} is proved in Lemma 4.1 of \cite{HuK3}. Here we will give a different simple proof  of this result. By \eqref{beta-lower-bd3}, 
\begin{equation}\label{c1-positive}
n-2-\frac{2m\alpha}{\beta}=\frac{n-2-(n+2)m}{\beta (1-m)}\left(\beta-\frac{2m\rho_1}{n-2-(n+2)m}\right)>0.
\end{equation}
By Corollary \ref{blow-up-self-similar-soln-decay-rate-cor} \eqref{z-growth-rate-bd} holds in $\mathbb{R}$. By  \eqref{z-eqn}, \eqref{z-growth-rate-bd} and \eqref{c1-positive},
\begin{equation}\label{z-ineqn6}
z_s+\beta e^{-\frac{\rho_1}{\beta}s}\widetilde{w}^{1-m}z\le\frac{\alpha}{\beta}\left(n-2-\frac{m\alpha}{\beta}\right)=C_2\quad\forall s\in\mathbb{R}.
\end{equation} 
Let $b_1>0$ and let $a(\widetilde{w},s)$ be given by \eqref{a-defn}. Then by multiplying \eqref{z-ineqn6} by $e^{a(\widetilde{w},s)}$ and integrating over $(-\infty,s)$,
\begin{equation}\label{z-representation-formula-upper-bd}
z(s)\le C_2\int_{-\infty}^se^{a(\widetilde{w},\rho)-a(\widetilde{w},s)}\,d\rho\quad\forall s\in\mathbb{R}.
\end{equation}
By  \eqref{blow-up-rate-x=0} and \eqref{z-growth-rate-bd},
\begin{equation}\label{tilde-w-lower-bd7}
\widetilde{w}(s)\ge\eta\quad\forall s\in\mathbb{R}.
\end{equation}
By \eqref{tilde-w-lower-bd7} and an argument similar to the proof of Lemma \ref{z-local-existence-lem}, we get \eqref{a-ineqn}.
By \eqref{z-defn}, \eqref{a-ineqn} and \eqref{z-representation-formula-upper-bd} we have
\begin{align}\label{z-upper-bd12}
r(\log w)_r=&z(s)\le C_2\int_{-\infty}^s\mbox{exp}\,\left(\beta\eta^{1-m}e^{-\frac{\rho_1}{\beta}s}(\rho-s)\right)\,d\rho=\frac{C_2\eta^{m-1}}{\beta}e^{\frac{\rho_1s}{\beta}}=\frac{C_2\eta^{m-1}}{\beta}r^{\frac{\rho_1}{\beta}}\quad\forall s\in\mathbb{R}\notag\\
\Rightarrow\quad(\log w)_r\le &\frac{C_2\eta^{m-1}}{\beta}r^{\frac{\rho_1}{\beta}-1}\quad\forall r>0
\end{align}
and \eqref{z-upper-bd2} follows. Integrating \eqref{z-upper-bd12} with respect to $r$ over $(0,r)$, by \eqref{blow-up-rate-x=0} we get \eqref{f-upper-bd11} and the proposition follows.

\end{proof}

\begin{prop}\label{f-lambda-representation-formula-prop}(cf. Corollary 2.6 of \cite{Hu3})
Let $n\ge 3$, $0<m<\frac{n-2}{n}$, $\rho_1>0$,   $\beta>\frac{m\rho_1}{n-2-nm}$ and $\alpha$  satisfies \eqref{alpha-beta-relation2}. Let $f_1\in C^2(0,\infty)$ be the  unique  solution of \eqref{f-ode} which satisfies \eqref{blow-up-rate-x=0} with $\eta=1$ given by Theorem \ref{blow-up-self-similar-soln-thm}. For any $\lambda>0$, let $f_{\lambda}\in C^2(0,\infty)$ be the  unique  solution of \eqref{f-ode}  which satisfies 
\begin{equation*}
\lim_{\rho\to 0}\rho^{\alpha/\beta}f_{\lambda}(\rho)=\lambda^{-\frac{\rho_1}{\beta(1-m)}}
\end{equation*}
given by Theorem \ref{blow-up-self-similar-soln-thm}. Then
\begin{equation}\label{f-lambda-representation-formula}
f_{\lambda}(r)=\lambda^{\frac{2}{1-m}}f_1(\lambda r)\quad\forall r>0.
\end{equation} 
\end{prop}
\begin{proof}
When $\beta\ge\frac{\rho_1}{n-2-nm}$, \eqref{f-lambda-representation-formula} is proved in Corollary 2.6 of \cite{Hu3}. Here we will extend this result for any $\beta>\frac{m\rho_1}{n-2-nm}$. Let
\begin{equation*}
v(r)=\lambda^{\frac{2}{1-m}}f_1(\lambda r)\quad\forall r>0.
\end{equation*}
Then $v$ satisfies \eqref{elliptic-eqn} and
\begin{equation*}
\lim_{r\to 0}r^{\alpha/\beta}v(r)=\lambda^{-\frac{\rho_1}{\beta(1-m)}}\lim_{r\to 0}(\lambda r)^{\alpha/\beta}f_1(\lambda r)
=\lambda^{-\frac{\rho_1}{\beta(1-m)}}.
\end{equation*}
Hence by Theorem \ref{blow-up-self-similar-soln-thm},
\begin{equation*}
v(r)=f_{\lambda}(r)\quad\forall r>0
\end{equation*}  
and \eqref{f-lambda-representation-formula} follows.

\end{proof}

\section{Existence and uniqueness of the singular self-similar solution of the logarithmic diffusion equation}
\setcounter{equation}{0}
\setcounter{thm}{0}

In this section we will use an argument similar to the proof of Theorem \ref{blow-up-self-similar-soln-thm} to prove Theorem \ref{log-blow-up-self-similar-soln-thm}. We assume that $n\ge 3$, $\eta>0$, $\rho_1>0$, $\beta_0>0$ and $\alpha_0=2\beta_0+\rho_1$ for the rest of the section. Suppose $g\in C^2(0,\infty)$ is a solution  of \eqref{g-ode} which satisfies \eqref{g-blow-up-rate-x=0}. Let $q(r)$  be given by 
\begin{equation}\label{q-defn}
q(r)=r^{\alpha/\beta}g(r)\quad\forall r=e^s,s\in\mathbb{R}
\end{equation}
and $h$ be given by
\begin{equation}\label{h-defn}
s=\log r,\quad h(s)=rq_r(r)/q(r).
\end{equation}  
Then $q\in C^2(0,\infty)$ and $h\in C^1(\mathbb{R})$. By the computation in \cite{Hs1} we get that $q$ satisfies
\begin{equation}\label{q-eqn}
\left(\frac{q_r}{q}\right)_r+\frac{n-1}{r}\cdot\frac{q_r}{q}+\beta_0r^{-1-\frac{\rho_1}{\beta_0}}q_r=\frac{\alpha_0(n-2)}{\beta_0r^2}.
\end{equation}
By \eqref{q-eqn}, $h$ satisfies
\begin{equation}\label{h-eqn}
h_s+(n-2)h+\beta_0 e^{-\frac{\rho_1s}{\beta_0}}\widetilde{q}h=\frac{\alpha_0(n-2)}{\beta_0}\quad\mbox{ in }\mathbb{R} 
\end{equation} 
where
\begin{equation}\label{tilde-w1-defn}
\widetilde{q}(s)=q(e^s)\quad\forall s\in\mathbb{R}.
\end{equation}
For any $b_1>0$ and function $q_1\in L_{loc}^{\infty}(\mathbb{R})$, let 
\begin{equation}\label{a0-defn}
a_0(q_1,s)=-\beta_0\int_s^{-b_1}e^{-\frac{\rho_1\rho}{\beta_0}}q_1(\rho)\,d\rho\quad\forall s\in\mathbb{R}. 
\end{equation}
If we assume that
\begin{equation*}
h(s_1)e^{a_0(\widetilde{q},s_1)+(n-2)s_1}\to 0\quad\mbox{ as }s_1\to-\infty
\end{equation*}
and
\begin{equation*}
\int_{-\infty}^se^{a_0(\widetilde{q},\rho)+(n-2)\rho}\,d\rho<\infty\quad\forall s\in\mathbb{R}
\end{equation*}
where $\widetilde{q}$ is given by \eqref{tilde-w1-defn}, then by multiplying  \eqref{h-eqn} by $e^{a_0(\widetilde{q},s)+(n-2)s}$ and integrating over $(-\infty,s)$, we get
\begin{equation}\label{h-representation-formula}
h(s)=\frac{\alpha_0(n-2)}{\beta_0}\int_{-\infty}^se^{a_0(\widetilde{q},\rho)-a_0(\widetilde{q},s)+(n-2)(\rho-s)}\,d\rho\quad\forall s\in\mathbb{R}. 
\end{equation}
This suggests one to use fixed point argument to prove the existence of radially symmetric solution $g$ of \eqref{log-elliptic-eqn} which satisfies \eqref{g-blow-up-rate-x=0}.

We will prove Theorem \ref{log-blow-up-self-similar-soln-thm} in several steps. We will first use fixed point argument to prove the local existence and uniqueness of  solution $h$ of \eqref{h-eqn} 
with
\begin{equation}\label{tilde-q-formula} 
\widetilde{q}(s)=\eta\, \mbox{exp}\left(\int_{-\infty}^sh(\rho)\,d\rho\right)
\end{equation}
which satisfies some decay condition for $h$. Then by an argument similar to the proof of Theorem \ref{z-existence-thm} one can conclude that this local solution can be extended to a unique global solution $h$ of \eqref{h-eqn} with $\widetilde{q}$ given by \eqref{tilde-q-formula} under the condition that $h$ satisfies some decay condition. We then remove the decay requirement on $q$ in the uniqueness of solution result. From these results on $h$ we get corresponding result for $g$ in Theorem \ref{log-blow-up-self-similar-soln-thm}.

\begin{lem}\label{h-local-existence-lem}
Let $n\ge 3$, $\eta>0$, $\rho_1>0$, $\beta_0>0$  and $\alpha_0=2\beta_0+\rho_1$.  Then there exists a constant $b_1>0$ such that the equation 
\eqref{h-eqn} has a unique  solution $h\in C^1(-\infty,-b_1)$   in $(-\infty,-b_1)$ that satisfies 
\begin{equation}\label{h-growth-rate-bd}
0<h(s)e^{-\rho_1s/\beta_0}\le C_0/\eta
\end{equation}
 in $(-\infty, -b_1)$ with $\widetilde{q}$ given by \eqref{tilde-q-formula} 
in $(-\infty, -b_1)$  where 
\begin{equation}\label{c0-defn}
C_0=\frac{\alpha_0(n-2)}{\beta_0^2}.
\end{equation}
\end{lem} 
\begin{proof}
We will prove this result by using fixed point argument. Let $b_1>0$. We define the Banach space 
\begin{equation*}
\cX_{b_1}'=\left\{(\widetilde{q},h): \widetilde{q},h\in C((-\infty,-b_1); \mathbb{R})\,\,\mbox{ such that }\,\,\|(\widetilde{q},h)\|_{\cX_{b_1}'}<\infty\right\}
\end{equation*}
with norm 
\begin{equation*}
\|(\widetilde{q},h)\|_{\cX_{b_1}'}=\max\left\{\| \widetilde{q}\|_{L^{\infty}\left((-\infty,-b_1);e^{-\frac{\rho_1s}{4\beta_0}}\right)},\|h\|_{L^{\infty}\left((-\infty,-b_1);e^{-\frac{\rho_1s}{2\beta_0}}\right)}\right\}.
\end{equation*}
Note that by an argument similar to the proof of \eqref{exp-term-bd} there exists a constant $C_0'>0$ such that
\begin{equation}\label{exp-term-bd2}
\left|\frac{\mbox{exp}\,\left(\frac{C_0\beta_0 a}{\rho_1\eta}\right)-1}{a\eta^{-1}}\right|\le C_0'\quad\forall 0<|a|\le 1.
\end{equation}
Let 
$$
\cD_{b_1}':=\left\{ (\widetilde{q},h)\in \cX_{b_1}:  \|(\widetilde{q},h)-(\eta,0)\|_{\cX_{b_1}'}\le\3_2\mbox{ and }\widetilde{q}(s)\ge\eta, 0\le h(s)e^{-\frac{\rho_1s}{\beta_0}}\le C_0\eta^{-1}\quad\forall s<-b_1 \right\}
$$
where $C_0$ is given by \eqref{c0-defn}
and 
\begin{equation*}\label{epsilon2-defn}
\3_2:=\frac{1}{2}\min(1,\eta)>0.
\end{equation*}
We now choose
\begin{equation}\label{b1-range2}
b_1>b_0':=\max \left(\frac{\beta_0}{\rho_1}\log\left(\frac{C_0'}{\3_2}\right),\frac{2\beta_0}{\rho_1}\log\left(\frac{C_0}{\3_2\eta}\right),\frac{\beta_0}{\rho_1}\left|\log\left(\frac{4\beta_0^2\eta}{3\rho_1}\right)\right|\right).
\end{equation} 
Then $(\eta,\min (C_0\eta^{-1},\3_2)e^{\rho_1s/\beta_0})\in \cD_{b_1}'$. Hence $\cD_{b_1}'\ne\phi$.
For any $(\widetilde{q},h)\in \cD_{b_1}',$ let 
\begin{equation*}
\Phi(\widetilde{q},h):=\left(\Phi_1(\widetilde{q},h),\Phi_2(\widetilde{q},h)\right)
\end{equation*}  
be given by 
\begin{equation}\label{phi12-defn5}
\left\{\begin{aligned}
&\Phi_1(\widetilde{q},h)(s):=\eta\, \mbox{exp}\left(\int_{-\infty}^sh(\rho)\,d\rho\right)\qquad\qquad\qquad\quad\,\,\forall s<-b_1\\
&\Phi_2(\widetilde{q},h)(s):=\frac{\alpha_0(n-2)}{\beta_0}\int_{-\infty}^se^{a_0(\widetilde{q},\rho)-a_0(\widetilde{q},s)+(n-2)(\rho-s)}\,d\rho\quad\forall s<-b_1
\end{aligned}\right.
\end{equation}
where  $a_0(\widetilde{q},s)$ are given by \eqref{a0-defn}. 
We first prove  that  $\Phi(\cD_{b_1}')\subset \cD_{b_1}'$. Let $(\widetilde{q},h)\in \cD_{b_1}'$. Then 
\begin{equation}\label{tilde-q-upper-lower-bd}
\eta\le\widetilde{q}(s)\le 3\eta/2\quad\forall s<-b_1
\end{equation}
and by the definition of $\Phi_1(\widetilde{q},h)$,
\begin{equation}\label{Phi1-ge-eta3}
\Phi_1(\widetilde{w},z)(s)\ge\eta\quad\forall s<-b_1.
\end{equation}
By \eqref{a0-defn}, \eqref{tilde-q-upper-lower-bd} and the mean value theorem for any $\rho<s<-b_1$ there exists a constant $\xi\in (\rho,s)$ such that
\begin{equation}\label{a0-ineqn}
a_0(\widetilde{q},\rho)-a_0(\widetilde{q},s)=\beta_0 e^{-\rho_1\xi/\beta_0}\widetilde{q}(\xi)(\rho-s)
\le\beta_0\eta e^{-\rho_1s/\beta_0}(\rho-s)\quad\forall \rho<s<-b_1.
\end{equation}
Hence by \eqref{b1-range2}, \eqref{phi12-defn5} and \eqref{a0-ineqn} we have
\begin{align}
&0<\Phi_2(\widetilde{q},h)(s)\le \frac{\alpha_0(n-2)}{\beta_0}\int_{-\infty}^s\mbox{exp}\left(\left(\beta_0\eta e^{-\rho_1s/\beta_0}+n-2\right)(\rho-s)\right)\,d\rho\le C_0\eta^{-1}e^{\rho_1s/\beta_0}\label{phi2-lower-bd5}\\
\Rightarrow\quad& \Phi_2(\widetilde{q},h)(s)e^{-\rho_1s/\beta_0}\le C_0\eta^{-1}\quad\forall s<-b_1\label{phi2-upper-bd6}\\
\Rightarrow\quad& \Phi_2(\widetilde{q},h)(s)e^{-\frac{\rho_1s}{2\beta_0}}\le C_0\eta^{-1}e^{-\frac{\rho_1b_1}{2\beta_0}}\quad\forall s<-b_1\notag\\
\Rightarrow\quad&\|\Phi_2(\widetilde{w},z)\|_{L^\infty\left((-\infty,-b_1);e^{-\frac{\rho_1s}{2\beta_0}}\right)}\le \3_2.\label{phi2-upper-bd7}
\end{align}
Now by \eqref{exp-term-bd2} and \eqref{b1-range2} we have
\begin{align}\label{phi1-eta-bd5}
\left|\Phi_1(\widetilde{w},z)(s)-\eta\right|
\le&\eta\left(\mbox{exp}\left(\int_{-\infty}^sh(\rho)\,d\rho\right)-1\right)\qquad\forall s<-b_1\notag\\
\le&\eta\left(\mbox{exp}\left(\int_{-\infty}^sC_0\eta^{-1}e^{\frac{\rho_1\rho}{\beta_0}}\,d\rho\right)-1\right)\quad\forall s<-b_1\notag\\
\le&\eta\left(\mbox{exp}\left(\frac{C_0\beta_0}{\rho_1\eta}e^{-\frac{\rho_1b_1}{\beta_0}}\right)-1\right)\quad\forall s<-b_1\notag\\
\le&C_0'e^{-\frac{\rho_1b_1}{\beta_0}}\qquad\qquad\qquad\qquad\forall s<-b_1\notag\\
\le&\3_2\qquad\qquad\qquad\qquad\qquad\forall s<-b_1.
\end{align}
By \eqref{Phi1-ge-eta3}, \eqref{phi2-lower-bd5}, \eqref{phi2-upper-bd6}, \eqref{phi2-upper-bd7},  and \eqref{phi1-eta-bd5}, we get $\Phi(\cD_{b_1}')\subset \cD_{b_1}'$. 

We will now show that when $b_1$ is sufficiently large, $\Phi :\cD_{b_1}'\to \cD_{b_1}'$ is a contraction map. Let $(\widetilde{q}_1,h_1), (\widetilde{q}_2,h_2)\in \cD_{b_1}'$ and $\delta:=\|(\widetilde{q}_1,h_1)-(\widetilde{q}_2,h_2)\|_{\cX_{b_1}'}$. Then 
\begin{equation}\label{tilde-qi-upper-lower-bd}
\eta\le\widetilde{q}_i(s)\le 3\eta/2\quad\forall s<b_1, i=1,2.
\end{equation}
By \eqref{a0-ineqn},
\begin{equation}\label{a0-ineqn2}
a_0(\widetilde{q}_i,\rho)-a_0(\widetilde{q}_i,s)\le\beta_0 e^{-\rho_1s/\beta_0}\eta(\rho-s)\quad\forall \rho<s<-b_1,i=1,2.
\end{equation}
By \eqref{tilde-qi-upper-lower-bd} and the mean value theorem there exists $\xi$ between $\int_{-\infty}^sh_1(\rho)\,d\rho$ and $\int_{-\infty}^sh_2(\rho)\,d\rho$ such that 
\begin{align}\label{phi1-12-l1-5}
|\Phi_1(\widetilde{q}_1,h_1)(s)-\Phi_1(\widetilde{q}_2,h_2)(s)|e^{-\frac{\rho_1 s}{4\beta_0}}=&\eta\left|\mbox{exp}\,\left(\int_{-\infty}^sh_1(\rho)\,d\rho\right)-\mbox{exp}\,\left(\int_{-\infty}^sh_2(\rho)\,d\rho\right)\right|e^{-\frac{\rho_1 s}{4\beta_0}}\notag\\
=&e^{-\frac{\rho_1 s}{4\beta_0}}\eta e^{\xi}\left|\int_{-\infty}^s(h_1(\rho)-h_2(\rho))\,d\rho\right|\notag\\
\le&e^{-\frac{\rho_1 s}{4\beta_0}}\max(\widetilde{q}_1(s),\widetilde{q}_2(s))\|h_1-h_2\|_{L^{\infty}\left((-\infty,-b_1);e^{-\frac{\rho_1s}{2\beta_0}}\right)}
\int_{-\infty}^se^{\frac{\rho_1\rho}{2\beta_0}}\,d\rho\notag\\
\le&\frac{3\beta_0\eta}{\rho_1}e^{-\frac{\rho_1b_1}{4\beta_0}}\delta
\qquad\forall s<-b_1.
\end{align}
Now by \eqref{phi12-defn5} we have
\begin{align}\label{phi2-difference6}
|\Phi_2(\widetilde{q}_1,h_1)(s)-\Phi_2(\widetilde{q}_2,h_2)(s)|
\le&\frac{\alpha_0(n-2)}{\beta_0}\int_{-\infty}^s\left|e^{a_0(\widetilde{q}_1,\rho)-a_0(\widetilde{q}_1,s)}-e^{a_0(\widetilde{q}_2,\rho)-a_0(\widetilde{q}_2,s)}\right|e^{(n-2)(\rho-s)}\,d\rho\notag\\
\le&\frac{\alpha_0(n-2)}{\beta_0}I_3\qquad\forall s<-b_1
\end{align}
where
\begin{equation*}
I_3=\int_{-\infty}^s\left|e^{a_0(\widetilde{q}_1,\rho)-a_0(\widetilde{q}_1,s)}-e^{a_0(\widetilde{q}_2,\rho)-a_0(\widetilde{q}_2,s)}\right|\,d\rho\quad\forall s<-b_1.
\end{equation*}
By \eqref{a0-ineqn2} and an argument similar to the proof of \eqref{exp-a-difference} we have
\begin{align}\label{exp-a0-difference}
&\left|e^{a_0(\widetilde{w}_1,\rho)-a_0(\widetilde{w}_1,s)}-e^{a_0(\widetilde{w}_2,\rho)-a_0(\widetilde{w}_2,s)}\right|\notag\\
\le&\frac{2\beta_0^2}{\rho_1}\left\{\mbox{exp}\,\left(\beta_0\eta e^{-\rho_1s/\beta_0}(\rho-s)-\frac{3\rho_1\rho}{4\beta_0}\right)
-\mbox{exp}\,\left(\beta_0\eta e^{-\rho_1s/\beta_0}(\rho-s)-\frac{3\rho_1s}{4\beta_0}\right)\right\}\delta\quad\forall \rho<s<-b_1.
\end{align}
Hence by \eqref{b1-range2} and \eqref{exp-a0-difference},
\begin{align}\label{I3-ineqn}
I_3\le&2\beta_0^2\rho_1^{-1}\delta\int_{-\infty}^s\left(\mbox{exp}\,\left(\beta_0\eta e^{-\rho_1s/\beta_0}(\rho-s)-\frac{3\rho_1\rho}{4\beta_0}\right)
-\mbox{exp}\,\left(\beta_0\eta e^{-\rho_1s/\beta_0}(\rho-s)-\frac{3\rho_1s}{4\beta_0}\right)\right)\,d\rho\notag\\
\le&2\beta_0^2\rho_1^{-1}\delta\left(\frac{e^{-\frac{3\rho_1s}{4\beta_0}}}{\beta_0\eta e^{-\rho_1s/\beta_0}-\frac{3\rho_1}{4\beta_0}}
-\frac{e^{-\frac{3\rho_1s}{4\beta_0}}}{\beta_0\eta e^{-\rho_1s/\beta_0}}\right)\notag\\
\le&\frac{3\eta^{-1}e^{\frac{5\rho_1s}{4\beta}}\delta}{\beta_0\eta
-\frac{3\rho_1}{4\beta_0}e^{\frac{\rho_1s}{\beta_0}}}\quad\forall s<-b_1\notag\\
\le&\frac{3\eta^{-1}e^{\frac{5\rho_1s}{4\beta_0}}\delta}{\beta_0\eta
-\frac{3\rho_1}{4\beta_0}e^{-\frac{\rho_1b_1}{\beta_0}}}\quad\forall s<-b_1.
\end{align}
Thus by \eqref{phi2-difference6} and \eqref{I3-ineqn} we have
\begin{equation}\label{phi2-12-l1-7}
|\Phi_2(\widetilde{q}_1,h_1)(s)-\Phi_2(\widetilde{q}_2,h_2)(s)|e^{-\frac{\rho_1s}{2\beta}}
\le\frac{3\alpha_0(n-2)\eta^{-1}e^{-\frac{3\rho_1b_1}{4\beta_0}}\delta}{\beta_0\left(\beta_0\eta
-\frac{3\rho_1}{4\beta_0}e^{-\frac{\rho_1b_1}{\beta_0}}\right)}
\quad\forall s<-b_1.
\end{equation}
By choosing $b_1$ sufficiently large in \eqref{phi1-12-l1-5} and \eqref{phi2-12-l1-7} we get
\begin{equation}\label{phi1-contraction7}
\|\Phi_1(\widetilde{q}_1,h_1)-\Phi_1(\widetilde{q}_2,h_2)\|_{L^\infty\left((-\infty,-b_1);e^{-\frac{\rho_1}{4\beta}s}\right)}\le\frac{\delta}{5}
\end{equation}
and
\begin{equation}\label{phi2-contraction7}
\|\Phi_2(\widetilde{q}_1,h_1)-\Phi_2(\widetilde{q}_2,h_2)\|_{L^\infty\left((-\infty,-b_1);e^{-\frac{\rho_1}{2\beta}s}\right)}\le\frac{\delta}{5}.
\end{equation}
By \eqref{phi1-contraction7} and \eqref{phi2-contraction7},  for sufficiently large $b_1$ we have
\begin{equation}\label{phi-contraction7}
\|\Phi(\widetilde{q}_1,h_1)-\Phi(\widetilde{q}_2,h_2)\|_{\cX_{b_1}'}\le\frac{\delta}{5}.
\end{equation}
Hence we can choose a  sufficiently large  $b_1$ such that the map $\Phi:\cD_{b_1}'\to \cD_{b_1}'$ is a contraction map with Lipschitz constant less than $1/5$. Since $\cD_{b_1}'$ is a complete metric space, by the contraction mapping theorem there exists a unique fixed point 
$(\widetilde{q},h)=\Phi(\widetilde{q},h)\in\cD_{b_1}'$ that satisfies 
\eqref{h-representation-formula},  \eqref{tilde-q-formula}  and
\begin{equation}\label{h-growth-rate-bd2}
\widetilde{q}(s)\ge\eta, \quad 0\le h(s)e^{-\frac{\rho_1}{\beta}s}\le C_0\eta^{-1}\quad\forall s<-b_1.
\end{equation}
Since $h=\Phi_2(\widetilde{q},h)$, by  \eqref{h-representation-formula}  $h(s)>0$ for any $s<-b_1$. This together with \eqref{h-growth-rate-bd2} implies that \eqref{h-growth-rate-bd} holds in $(-\infty,-b_1)$. Multiplying \eqref{h-representation-formula} by $e^{a_0(\widetilde{q},s)+(n-2)s}$ and differentiating  with respect to $s$ we get that $h$ satisfies \eqref{h-eqn} in $(-\infty,-b_1)$.

Finally by an argument similar to the proof of Lemma \ref{z-local-existence-lem} $h$ is unique
and the lemma follows.

\end{proof}

By an argument similar to the proof of Theorem \ref{z-existence-thm} but with Lemma \ref{h-local-existence-lem} replacing Lemma \ref{z-local-existence-lem} in the proof we get the following theorem.

\begin{thm}\label{h-existence-thm}
Let $n\ge 3$, $\eta>0$, $\rho_1>0$, $\beta_0>0$, $\alpha_0=2\beta_0+\rho_1$ and $C_0$ be given by \eqref{c0-defn}. Then  the equation \eqref{h-eqn} has a unique  solution $h\in C^1(\mathbb{R})$   in $\mathbb{R}$ that satisfies \eqref{h-growth-rate-bd} in $\mathbb{R}$ with $\widetilde{q}$ given by
\eqref{tilde-q-formula}. 

\end{thm} 
By an argument similar to the proof of Theorem \ref{blow-up-self-similar-soln-thm} but with Theorem \ref{h-existence-thm} replacing Theorem \ref{z-existence-thm} in the proof, we get Theorem \ref{log-blow-up-self-similar-soln-thm} and the following corollary.

\begin{cor}\label{h-decay-cor}
Let $n\ge 3$, $\eta>0$, $\rho_1>0$, $\beta_0>0$  and $\alpha_0$ be given by \eqref{alpha0-beta0-relation}. Let $g\in C^2(0,\infty)$ be the unique solution of \eqref{g-ode} that satisfies \eqref{g-blow-up-rate-x=0} given by Theorem \ref{log-blow-up-self-similar-soln-thm} and let $h$ be given by \eqref{h-defn}. Then
$h$ satisfies \eqref{h-growth-rate-bd} in $\mathbb{R}$.
\end{cor} 

By an argument similar to the proof of  Corollary \ref{f-upper-lower-bds-cor} we get  the following corollary.

\begin{cor}\label{g-upper-lower-bds-cor}
Let $n\ge 3$, $\eta>0$, $\rho_1>0$, $\beta_0>0$, $\alpha_0=2\beta_0+\rho_1$ and $C_0$ be given by \eqref{c0-defn}. Let $g\in C^2(0,\infty)$ be the unique solution of \eqref{g-ode} that satisfies \eqref{g-blow-up-rate-x=0} given by Theorem \ref{log-blow-up-self-similar-soln-thm} which satisfies \eqref{g-blow-up-rate-x=0} 
 where $q$, $h$, are given by \eqref{q-defn}, \eqref{h-defn}. Then
\begin{equation}\label{g-gr-ineqn}
g+\frac{\beta_0}{\alpha_0}rg_r>0\quad\forall r>0
\end{equation}
and
\begin{equation}\label{g-upper-lower-bds}
\eta r^{-\alpha_0/\beta_0}<g(r)\le \eta r^{-\alpha_0/\beta_0}\mbox{exp}\,\left(\frac{\alpha_0(n-2)}{\beta_0\eta}r^{\rho_1/\beta_0}\right)\quad\forall r>0.
\end{equation}
\end{cor} 

By Corollary \ref{g-upper-lower-bds-cor} and an argument similar to the proof of Corollary \ref{f-upper-lower-bds-cor} we have the following result.

\begin{cor}\label{gr-negative2-cor}
Let $n\ge 3$, $\eta>0$, $\rho_1>0$, $\beta_0>0$, $\alpha_0=2\beta_0+\rho_1$ and $C_0$ be given by \eqref{c0-defn}. Let $g\in C^2(0,\infty)$ be the unique solution of \eqref{g-ode} that satisfies \eqref{g-blow-up-rate-x=0} given by Theorem \ref{log-blow-up-self-similar-soln-thm} which satisfies \eqref{g-blow-up-rate-x=0}  where $q$, $h$, are given by \eqref{q-defn}, \eqref{h-defn}. Then
\begin{equation}\label{gr-negative2}
g_r<0\quad\forall r>0.
\end{equation}
\end{cor}

By an argument similar to the proof of Proposition \ref{f-lambda-representation-formula-prop} but with Theorem \ref{log-blow-up-self-similar-soln-thm} replacing Theorem \ref{blow-up-self-similar-soln-thm} in the proof there we have the following result.

\begin{prop}\label{g-lambda-representation-formula-prop}
Let $n\ge 3$, $\eta>0$, $\rho_1>0$, $\beta_0>0$, $\alpha_0=2\beta_0+\rho_1$ and $C_0$ be given by \eqref{c0-defn}.  Let $g\in C^2(0,\infty)$ be the unique solution of \eqref{g-ode} that satisfies \eqref{g-blow-up-rate-x=0} given by Theorem \ref{log-blow-up-self-similar-soln-thm} which satisfies \eqref{g-blow-up-rate-x=0} with $\eta=1$ with $\eta=1$ where $q$, $h$, are given by \eqref{q-defn}, \eqref{h-defn}.  For any $\lambda>0$, let $g_{\lambda}\in C^2(0,\infty)$ be the  unique   solution of \eqref{g-ode} given by Theorem \ref{log-blow-up-self-similar-soln-thm} which satisfies 
\begin{equation*}
\lim_{\rho\to 0}\rho^{\alpha_0/\beta_0}g_{\lambda}(\rho)=\lambda^{-\rho_1/\beta_0}.
\end{equation*}
Then
\begin{equation*}\label{g-lambda-representation-formula}
g_{\lambda}(r)=\lambda^2g_1(\lambda r)\quad\forall r>0.
\end{equation*} 
\end{prop}

\section{Asymptotic  of the singular self-similar solution of the fast diffusion equation as $|x|\to\infty$}
\setcounter{equation}{0}
\setcounter{thm}{0}

In this section we will use a modification of the technique of \cite{Hs4} to prove the asymptotic behaviour of the singular self-similar solution of the fast diffusion equation as $|x|\to\infty$.
We assume that  $n\ge 3$, $0<m<\frac{n-2}{n}$, $\rho_1>0$, $\eta>0$, $\beta>\frac{m\rho_1}{n-2-nm}$ and $\alpha$ satisfies \eqref{alpha-beta-relation2} and  let $f$ be the  unique  radially symmetric solution  of \eqref{elliptic-eqn} that satisfies \eqref{blow-up-rate-x=0} and \eqref{z-growth-rate-bd} in $\mathbb{R}$ given by Theorem \ref{blow-up-self-similar-soln-thm} where $w$ and $z$ are given by \eqref{w-defn} and \eqref{z-defn} for the rest of the section. Let $B_1(0)=\{x\in\mathbb{R}^n:|x|<1\}$,
\begin{equation}\label{v-defn}
v(r)=r^2f(r)^{1-m}\quad\forall r>0,
\end{equation}
and
\begin{equation*}
\beta_0'=\frac{m\rho_1}{n-2-nm},\quad\beta_1=\frac{\rho_1}{n-2-nm}.
\end{equation*}
Note that
\begin{equation}\label{alpha>n-beta}
\alpha> n\beta\quad (\alpha=n\beta, \alpha\le n\beta) \quad\Leftrightarrow\quad\beta<\beta_1 \quad \left(\beta=\beta_1,\beta>\beta_1,\,respectively\right)
\end{equation}
and
\begin{equation}\label{alpha-beta-m-relation}
\beta>\beta_0'\quad\Leftrightarrow\quad \frac{m\alpha}{\beta}<n-2.
\end{equation}
Hence when \eqref{alpha-beta-m-relation} holds, there exists a constant $\3_3>0$ such that
\begin{equation}\label{alpha-beta-lower-bd}
\frac{\alpha}{\beta}<\frac{n-2}{m}-\3_3.
\end{equation}
On the other hand multiplying \eqref{f-ode} by $r^{n-1}$ and integrating over $(\xi,r)$, $r>\xi>0$,
\begin{align}
&r^{n-1}(f^m/m)_r(r)=c(\xi)-\beta r^nf(r)-(\alpha-n\beta)\int_{\xi}^r\rho^{n-1}f(\rho)\,d\rho\quad\forall r>\xi>0\label{f-intregal-eqn}\\
\Rightarrow\quad&(f^m/m)_r(r)=c(\xi)r^{1-n}-\beta rf(r)-\frac{\alpha-n\beta}{r^{n-1}}\int_{\xi}^r\rho^{n-1}f(\rho)\,d\rho\quad\forall r>\xi>0\label{f-intregal-eqn30}
\end{align}
where 
\begin{equation}\label{c-xi-defn}
c(\xi)=\xi^{n-1}f(\xi)^{m-1}f_r(\xi)+\beta \xi^nf(\xi).
\end{equation}
Note that  by \eqref{f-intregal-eqn},
\begin{equation}\label{rfr-f-eqn}
r^{n-2}f^m(r)\cdot \frac{rf_r(r)}{f(r)}=c(\xi)-\beta r^nf(r)-(\alpha-n\beta)\int_{\xi}^r\rho^{n-1}f(\rho)\,d\rho\quad\forall r>\xi>0.
\end{equation}
Integrating \eqref{f-intregal-eqn30} over $(r,r_i)$, $0<\xi<r<r_i$, $i\in\mathbb{Z}^+$, we get
\begin{equation}\label{f-intregal-eqn4}
\frac{f(r)^m}{m}=\frac{f(r_i)^m}{m}+\frac{c(\xi)}{n-2}(r_i^{2-n}-r^{2-n})+\beta\int_r^{r_i}\rho f(\rho)\,d\rho+\int_r^{r_i}\frac{\alpha-n\beta}{s^{n-1}}\left(\int_{\xi}^s\rho^{n-1}f(\rho)\,d\rho\right)\,ds
\end{equation}
for any $r_i>r>\xi>0$, $i\in\mathbb{Z}^+$.
Note that for $\beta_0'<\beta\le\beta_1$, if $r_i\to\infty$ as $i\to\infty$ and 
\begin{equation}\label{fri-goes-to0}
f(r_i)\to 0\quad\mbox{  as }i\to\infty,
\end{equation}
 then by letting $i\to\infty$ in \eqref{f-intregal-eqn4}, by  \eqref{fri-goes-to0} and the monotone convergence theorem we get,
\begin{equation}\label{f-intregal-eqn5}
\frac{f(r)^m}{m}=-\frac{c(\xi)}{n-2}r^{2-n}+\beta\int_r^{\infty}\rho f(\rho)\,d\rho+\int_r^{\infty}\frac{\alpha-n\beta}{s^{n-1}}\left(\int_{\xi}^s\rho^{n-1}f(\rho)\,d\rho\right)\,ds\quad\forall r>\xi>0.
\end{equation}
We now recall a result of \cite{Hu3}.

\begin{lem}\label{f-integral-eqn-lemma}(Lemma 2.1 of \cite{Hu3})
Let $\beta\ge\beta_1$ and $f\in C^2(\mathbb{R}^n\setminus\{0\})$ be a radially symmetric solution of \eqref{elliptic-eqn} in $\mathbb{R}^n\setminus\{0\}$ which satisfies \eqref{blow-up-rate-x=0}. Then $f$ satisfies
\begin{equation*}
r^{n-1}(f^m/m)'(r)+\beta r^nf(r)=(n\beta-\alpha)\int_0^rf(\rho)\rho^{n-1}\,d\rho\quad\forall r>0
\end{equation*}
if $\beta>\beta_1$ and $f$ satisfies
\begin{equation*}
r^{n-1}(f^m/m)'(r)+\beta r^nf(r)=\beta\eta\quad\forall r>0\quad\mbox{ if }\beta=\beta_1.
\end{equation*}
\end{lem}

By Lemma \ref{f-integral-eqn-lemma} and the same argument as in \cite{Hs4} we have the following result.

\begin{prop}\label{beta>beta1-limit-case-prop}
Let $n\ge 3$, $0<m<\frac{n-2}{n}$, $\rho_1>0$, $\beta>\beta_1$ and $\alpha=\alpha_m$ be given by \eqref{alpha-beta-relation2}. Let $f\in C^2(\mathbb{R}^n\setminus\{0\})$ be the 
unique  radially symmetric solution  of \eqref{elliptic-eqn} that satisfies \eqref{blow-up-rate-x=0}  given by Theorem \ref{blow-up-self-similar-soln-thm}. Then $f$ satisfies \eqref{f-limit-infty}.
\end{prop}

Hence the result of Theorem \ref{soln-at-x-infty-thm} for the case $\beta>\beta_1$ is included in Proposition
\ref{beta>beta1-limit-case-prop}. Thus we will now only need to consider the case $\beta_0'<\beta\le\beta_1$ for the rest of this section. We will first prove that the function $v(r)$ is uniformly bounded above and below on $[1,\infty)$. Then we will write $v$ using an integral representation formula and 
derive the limit behaviour of $v(r)$ as $r\to\infty$ using this integral representation formula of $v(r)$ and the l'Hosiptal rule. We start will a technical lemma.

\begin{lem}\label{vr-positive-lem}
Suppose that $\beta_0'<\beta\le\beta_1$. Let $v$ be given by \eqref{v-defn},
\begin{equation}\label{m1-defn}
M_1=\frac{1}{2(1-m)}\left(\frac{\alpha-n\beta}{m\3_3}+\beta\right)^{-1}
\end{equation}
and
\begin{equation}\label{b-set-defn}
\mathcal{B}=\left\{r>1: v(r)\le M_1\right\}.
\end{equation}
Suppose
\begin{equation}\label{set-b-not-empty-infty}
[R_1,\infty)\cap\mathcal{B}\ne\phi\quad\forall R_1>1.
\end{equation}
Then there exists a constant $R_0>1$ such that
\begin{equation}\label{vr-positive}
v_r(r)>0\quad\forall r\in [R_0,\infty)\cap\mathcal{B}.
\end{equation} 
\end{lem}
\begin{proof}
Suppose \eqref{vr-positive} does not hold. Then there exists a sequence $\{r_i\}_{i=1}^{\infty}\subset\mathcal{B}$, $r_i\to\infty$ as $i\to\infty$, such that
\begin{equation}\label{vr-neg-sequence}
v_r(r_i)\le 0\quad\forall i\in\mathbb{Z}^+.
\end{equation}
Then there exists a subsequence of $\{r_i\}_{i=1}^{\infty}$ which we may assume without loss of generality to be the sequence itself such that 
\begin{equation*}
a_1:=\lim_{i\to\infty}r_i^{n-2}f^m(r_i)\in [0,\infty]\quad\mbox{ exists. }
\end{equation*}
Note that since $\{r_i\}_{i=1}^{\infty}\subset\mathcal{B}$,
\begin{align}
&r_i^2f(r_i)^{1-m}=v(r_i)\le M_1\quad\forall i\in\Z^+\label{fri-limit-0}\\
\Rightarrow\quad&0<f(r_i)\le M_1r_i^{-\frac{2}{1-m}}\to 0\quad\mbox{ as }i\to\infty\notag
\end{align}
and \eqref{fri-goes-to0} holds. Hence \eqref{f-intregal-eqn5}  holds. 
By \eqref{f-fr-ineqn} of Corollary \ref{f-upper-lower-bds-cor}, Corollary  \ref{fr-negative2-cor} and \eqref{alpha-beta-lower-bd},
\begin{equation}\label{rfr-f-bd}
0>\frac{rf_r(r)}{f(r)}>-\frac{\alpha}{\beta}>-\frac{n-2}{m}+\3_3\quad\forall r>0.
\end{equation}
By \eqref{fri-limit-0} and \eqref{rfr-f-bd} the sequence $\{r_i\}_{i=1}^{\infty}$ has a subsequence  which we may assume without loss of generality to be the sequence itself such that 
\begin{equation}\label{a2-defn}
a_2:=\lim_{i\to\infty}\frac{r_if_r(r_i)}{f(r_i)}\quad\mbox{ exists and } a_2\in \left[-\frac{n-2}{m}+\3_3,0\right]
\end{equation}
and 
\begin{equation}\label{a3-defn}
a_3:=\lim_{i\to\infty}r_i^2f(r_i)^{1-m}\quad\mbox{ exists and } a_3\in [0,M_1].
\end{equation}
We now divide the proof into two cases.

\noindent{\bf Case 1}: $0\le a_1<\infty$.

\noindent 
Note that
\begin{equation}\label{v-limit-infty}
\lim_{i\to\infty}v(r_i)=\lim_{i\to\infty}r_i^2f(r_i)^{1-m}=\lim_{i\to\infty}r_i^{-\frac{n-2-nm}{m}}\cdot\lim_{i\to\infty}(r_i^{n-2}f^m(r_i))^{\frac{1-m}{m}}=0\cdot a_1^{\frac{1-m}{m}}=0.
\end{equation}
By \eqref{rfr-f-eqn},
\begin{equation*}
\lim_{i\to\infty}r_i^{n-2}f^m(r_i)\cdot\lim_{i\to\infty} \frac{r_if_r(r_i)}{f(r_i)}=c(\xi)-\beta\lim_{i\to\infty}r_i^{n-\frac{n-2}{m}}\cdot\lim_{i\to\infty}(r_i^{n-2} f(r_i)^m)^{\frac{1}{m}}-(\alpha-n\beta)\lim_{i\to\infty}\int_{\xi}^{r_i}\rho^{n-1}f(\rho)\,d\rho
\end{equation*}
holds for any $\xi>0$. Hence
\begin{align}
&a_1a_2=c(\xi)-(\alpha-n\beta)\int_{\xi}^{\infty}\rho^{n-1}f(\rho)\,d\rho\quad\forall\xi>0\label{a12-relation}\\
\Rightarrow\quad&0<\int_1^{\infty}\rho f(\rho)\,d\rho<\int_1^{\infty}\rho^{n-1}f(\rho)\,d\rho<\infty.\label{f-integrals-finite}
\end{align}
Then by \eqref{f-intregal-eqn5}, \eqref{f-integrals-finite} and the l'Hospital rule,
\begin{align}\label{f-limit-eqn1}
\frac{a_1}{m}=&\frac{1}{m}\lim_{i\to\infty}r_i^{n-2}f(r_i)^m=-\frac{c(\xi)}{n-2}+\beta\lim_{i\to\infty}\frac{\int_{r_i}^{\infty}\rho f(\rho)\,d\rho}{r_i^{2-n}}+\lim_{i\to\infty}\frac{\int_{r_i}^{\infty}\frac{\alpha-n\beta}{s^{n-1}}\left(\int_{\xi}^s\rho^{n-1}f(\rho)\,d\rho\right)\,ds}{r_i^{2-n}}\notag\\
=&\frac{1}{n-2}\left(-c(\xi)+\beta\lim_{i\to\infty}\frac{r_i f(r_i)}{r_i^{1-n}}+\lim_{i\to\infty}\frac{\frac{\alpha-n\beta}{r_i^{n-1}}\left(\int_{\xi}^{r_i}\rho^{n-1}f(\rho)\,d\rho\right)}{r_i^{1-n}}\right)\notag\\
=&\frac{1}{n-2}\left(-c(\xi)+\beta\lim_{i\to\infty}r_i^{n-\frac{n-2}{m}}\cdot\lim_{i\to\infty}(r_i^{n-2} f(r_i)^m)^{\frac{1}{m}}+(\alpha-n\beta)\int_{\xi}^{\infty}\rho^{n-1}f(\rho)\,d\rho\right)\notag\\
=&\frac{1}{n-2}\left(-c(\xi)+(\alpha-n\beta)\int_{\xi}^{\infty}\rho^{n-1}f(\rho)\,d\rho\right).
\end{align}
We now divide this case into two subcases. 

\noindent{\bf Subcase A}: $0<a_1<\infty$.

\noindent 
By \eqref{a12-relation} and \eqref{f-limit-eqn1},
\begin{equation*}
\frac{a_1}{m}=-\frac{a_1a_2}{n-2}\quad\Rightarrow\quad a_2=-\frac{n-2}{m}
\end{equation*}
which contradicts \eqref{a2-defn}. Hence subcase A does not hold.

\noindent{\bf Subcase B}: $a_1=0$.

\noindent By \eqref{c-xi-defn}, \eqref{v-limit-infty}, \eqref{f-integrals-finite}, \eqref{f-limit-eqn1} and the l'Hospital rule,
\begin{align}\label{rfrf-ratio-limit-zero}
(\alpha-n\beta)\int_{\xi}^{\infty}\rho^{n-1}f(\rho)\,d\rho=&c(\xi)=\xi^{n-1}f(\xi)^{m-1}f_r(\xi)+\beta \xi^nf(\xi)\quad\forall \xi>0\notag\\
\Rightarrow\qquad\qquad\qquad\quad\lim_{i\to\infty}\frac{r_if_r(r_i)}{f(r_i)}=&-\beta\lim_{i\to\infty} r_i^2f(r_i)^{1-m}+(\alpha-n\beta)\lim_{i\to\infty}\frac{\int_{r_i}^{\infty}\rho^{n-1}f(\rho)\,d\rho}{r_i^{n-2}f^m(r_i)}\notag\\
=&(n\beta-\alpha)\lim_{i\to\infty}\frac{r_i^{n-1}f(r_i)}{(n-2)r_i^{n-3}f(r_i)^m+mr_i^{n-2}f(r_i)^{m-1}f_r(r_i)}\notag\\
=&(n\beta-\alpha)\lim_{i\to\infty}\frac{r_i^2f(r_i)^{1-m}}{(n-2)+m(r_if_r(r_i)/f(r_i))}\notag\\
=&\frac{n\beta-\alpha}{n-2+ma_2}\lim_{i\to\infty}r_i^2f(r_i)^{1-m}\notag\\
=&0.
\end{align}
By \eqref{rfrf-ratio-limit-zero} there exists $i_0\in\mathbb{Z}^+$ such that
\begin{equation}\label{rfrf-lower-bd5}
\frac{r_if_r(r_i)}{f(r_i)}>-\frac{1}{1-m}\quad\forall i\ge i_0.
\end{equation}
Since
\begin{equation}\label{vr-expression}
v_r(r)=\frac{v(r)}{r}\left(1+\frac{1-m}{2}\cdot\frac{rf_r(r)}{f(r)}\right)\quad\forall r>0,
\end{equation}
by \eqref{rfrf-lower-bd5},
\begin{equation}\label{vr-positive3}
v_r(r_i)\ge\frac{v(r_i)}{2r_i}>0\quad\forall  i\ge i_0
\end{equation}
which contradicts \eqref{vr-neg-sequence}. Hence subcase B does not hold.

Thus case 1 does not hold.

\noindent{\bf Case 2}: $a_1=\infty$.

\noindent 
By \eqref{c-xi-defn}, \eqref{rfr-f-eqn}, \eqref{a2-defn} and \eqref{a3-defn},
\begin{align}\label{rfr-f-limit-eqn2}
\lim_{i\to\infty}\frac{r_if_r(r_i)}{f(r_i)}=&\lim_{i\to\infty}\frac{f(1)^{m-1}f_r(1)+\beta f(1)}{r_i^{n-2}f^m(r_i)}-\beta \lim_{i\to\infty}r_i^2f^{1-m}(r_i)-(\alpha-n\beta)\lim_{i\to\infty}\frac{\int_1^{r_i}\rho^{n-1}f(\rho)\,d\rho}{r_i^{n-2}f^m(r_i)}\notag\\
=&-\beta a_3-(\alpha-n\beta)\lim_{i\to\infty}\frac{\int_1^{r_i}\rho^{n-1}f(\rho)\,d\rho}{r_i^{n-2}f^m(r_i)}.
\end{align}
Hence if $f\in L^1(\mathbb{R}^n\setminus B_1(0))$, then by \eqref{m1-defn}, \eqref{a3-defn} and \eqref{rfr-f-limit-eqn2},
\begin{equation}\label{rfrf-ratio-limit}
\lim_{i\to\infty}\frac{r_if_r(r_i)}{f(r_i)}=-\beta a_3\ge-\beta M_1\ge-\frac{1}{2(1-m)}.
\end{equation}
If $f\not\in L^1(\mathbb{R}^n\setminus B_1(0))$, then by \eqref{m1-defn}, \eqref{a2-defn}, \eqref{a3-defn}, \eqref{rfr-f-limit-eqn2}  and the l'Hospital rule,
\begin{align}\label{rfrf-ratio-limit2}
\lim_{i\to\infty}\frac{r_if_r(r_i)}{f(r_i)}=&-\beta a_3-(\alpha-n\beta)\lim_{i\to\infty}\frac{r_i^{n-1}f(r_i)}{(n-2)r_i^{n-3}f(r_i)^m+mr_i^{n-2}f(r_i)^{m-1}f_r(r_i)}\notag\\
=&-\beta a_3-(\alpha-n\beta)\lim_{i\to\infty}\frac{r_i^2f(r_i)^{1-m}}{(n-2)+m(r_if_r(r_i)/f(r_i))}\notag\\
=&- \left(\beta+\frac{(\alpha-n\beta)}{m}\cdot\frac{1}{\frac{n-2}{m}+a_2}\right)a_3\notag\\
\ge&- \left(\beta+\frac{(\alpha-n\beta)}{m\3_3}\right)M_1\notag\\
=&-\frac{1}{2(1-m)}.
\end{align}
By \eqref{rfrf-ratio-limit} and \eqref{rfrf-ratio-limit2} there exists $i_0\in\mathbb{Z}^+$ such that
\eqref{rfrf-lower-bd5} holds. Then by \eqref{rfrf-lower-bd5} and \eqref{vr-expression} we get
\eqref{vr-positive3} which contradicts \eqref{vr-neg-sequence}.
Hence case 2 does not hold. 

Hence by case 1 and case 2 no such sequence $\{r_i\}_{i=1}^{\infty}$ exists.
Thus there exists a constant $R_0>0$ such that
\eqref{vr-positive} holds.

\end{proof}

\begin{lem}\label{v-lower-bd-lem}
Suppose that $\beta_0'<\beta\le\beta_1$. Let $v$ be given by \eqref{v-defn}. Then there exists a constant $C_6>0$ such that
\begin{equation}\label{v-lower-bd}
v(r)\ge C_6\quad\forall r\ge 1.
\end{equation} 
\end{lem}
\begin{proof}
Let $M_1$ and $\mathcal{B}$ be given by \eqref{m1-defn} and \eqref{b-set-defn}. 
We now divide the proof into two cases.

\noindent{\bf Case 1}: $\{r\ge R_1:v(r)<M_1\}\ne\phi\quad\forall R_1>1$.

\noindent Then by Lemma \ref{vr-positive-lem} there exists a constant $R_0>1$ such that
\eqref{vr-positive} holds. By the hypothesis of case 1 there exists $r_1\in\{r\ge R_0:v(r)<M_1\}$. Then by Lemma \ref{vr-positive-lem},
\begin{equation*}
v_r(r_1)>0\quad\mbox{ and }\quad v(r_1)<M_1.
\end{equation*}
Hence there exists a constant $\delta_1>0$ such that
\begin{equation*}
v(r_1)<v(r)<M_1\quad\forall r_1<r<r_1+\delta_1.
\end{equation*}
Let $r_2=\sup\{r_0>r_1:v(r_1)<v(r)<M_1\quad\forall r_1<r<r_0\}$. Then $r_2>r_1$. Suppose $r_2<\infty$. Then $v(r_2)\le M_1$. If $v(r_2)<M_1$, then by \eqref{vr-positive},
\begin{equation}\label{vr-positive7}
v_r(r)>0\quad\forall r_1<r\le r_2.
\end{equation}
Hence there exists $\delta_2>0$ such that
\begin{equation*}
v_r(r)>0\quad\forall r_1<r\le r_2+\delta_2
\end{equation*}
and
\begin{equation*}
v(r_1)<v(r)<M_1\quad\forall r_1<r<r_2+\delta_2
\end{equation*}
which contradicts the choice of $r_2$. Hence $v(r_2)=M_1$ and by \eqref{vr-positive} $v_r(r_2)>0$. 
Then there exists a constant $\delta_3>0$ such that
\begin{equation*}
v_r(r)>0\quad\forall r_1<r\le r_2+\delta_3\quad\mbox{ and }\quad v(r)>M_1\quad\forall r_2<r<r_2+\delta_3.
\end{equation*}
By the hypothesis in case 1 we can choose $r_3>r_2+\delta_3$ such that $v(r_3)<M_1$. Let $r_4=\sup\{r_0>r_2:v(r)>M_1\quad\forall r_2<r<r_0\}$. Then $r_2+\delta_3\le r_4<r_3$,
\begin{equation*}
v(r)>M_1\quad\forall r_2<r<r_4\quad\mbox{ and }\quad v(r_4)=M_1.
\end{equation*} 
Hence $v_r(r_4)\le 0$ which contradicts \eqref{vr-positive}. Hence $r_2=\infty$. Then \eqref{v-lower-bd} holds with $C_6=\min_{1\le r\le r_1}v(r)$.

\noindent{\bf Case 2}: There exists a constant $R_1>0$ such that $\{r\ge R_1:v(r)<M_1\}=\phi$. 

\noindent Then \eqref{v-lower-bd} holds with $C_6=\min (M_1,\min_{1\le r\le R_1}v(r))$ and the lemma follows.

\end{proof}

\begin{lem}\label{v-upper-bd-lem}
Suppose that $\beta_0'<\beta\le\beta_1$. Let $v$ be given by \eqref{v-defn}. Then there exists a constant $C_7>0$ such that
\begin{equation}\label{v-upper-bd}
v(r)\le C_7\quad\forall r\ge 1.
\end{equation} 
\end{lem}
\begin{proof}
Let $a_0=n-\frac{2}{1-m}>0$ and let $C_6>0$ be as given by Lemma \ref{v-lower-bd-lem}. 
By Corollary \ref{fr-negative2-cor}, \eqref{fr-negative2} holds. Hence by 
 \eqref{fr-negative2}, \eqref{alpha>n-beta}, \eqref{f-intregal-eqn30}, \eqref{c-xi-defn} and Lemma \ref{v-lower-bd-lem},
\begin{align}\label{f-intregal-eqn2}
&f(r)^{m-1}f_r(r)\le C(1)r^{1-n}-\beta rf(r)-\frac{\alpha-n\beta}{r^{n-1}}\int_1^r\rho^{n-1}f(r)\,d\rho=C(1)r^{1-n}-\frac{\alpha}{n}rf(r)\quad\forall r\ge 1\notag\\
\Rightarrow\quad&f(r)^{m-2}f_r(r)+\frac{\alpha}{n}r\le \frac{C(1)r^{1-a_0}}{(r^2f^{1-m}(r))^{\frac{1}{1-m}}}\le C_8r^{1-a_0}\quad\forall r\ge 1.
\end{align}
where $C_8=C(1)/C_6^{\frac{1}{1-m}}$. We now choose $R_1>\max\left(1,(2nC_8/\alpha)^{1/a_0}\right)$. Then
\begin{equation}\label{ineqn10}
C_8r^{1-a_0}\le\frac{\alpha}{2n}R_1^{a_0}r^{1-a_0}\le\frac{\alpha}{2n}r\quad\forall r\ge R_1.
\end{equation}
By \eqref{f-intregal-eqn2} and \eqref{ineqn10},
\begin{equation}\label{f-intregal-eqn3}
f(r)^{m-2}f_r(r)+\frac{\alpha}{2n}r\le 0\quad\forall r\ge R_1.
\end{equation}
Integrating \eqref{f-intregal-eqn3} over $(R_1,r)$, $r\ge 2R_1$,
\begin{align*}
&f^{m-1}(R_1)+\frac{\alpha (1-m)}{4n}r^2\le f^{m-1}(r)+\frac{\alpha (1-m)}{4n}R_1^2\le f^{m-1}(r)+\frac{\alpha (1-m)}{16n}r^2\quad\forall r\ge 2R_1\notag\\
\Rightarrow\quad&f^{m-1}(R_1)+\frac{3\alpha (1-m)}{16n}r^2\le f^{m-1}(r)\qquad\qquad\qquad\forall r\ge 2R_1\notag\\
\Rightarrow\quad&r^2f^{1-m}(r)\le\frac{r^2}{f^{m-1}(R_1)+\frac{3\alpha (1-m)}{16n}r^2}\le\frac{16n}{3\alpha (1-m)}\quad\,\forall r\ge 2R_1.
\end{align*}
Then \eqref{v-upper-bd} holds with 
\begin{equation*}
C_7=\max\left(\frac{16n}{3\alpha (1-m)},\max_{1\le r\le 2R_1}v(r)\right).
\end{equation*}

\noindent{\ni{\it Proof of Theorem \ref{soln-at-x-infty-thm}:}}
The result for the case $\beta>\beta_1$ is given by Proposition \ref{beta>beta1-limit-case-prop}.
Hence we can now assume that $\beta_0'<\beta\le\beta_1$. By  Lemma \ref{v-lower-bd-lem} and Lemma \ref{v-upper-bd-lem} there exist constants $C_6>0$, $C_7>0$, such that \eqref{v-lower-bd} and \eqref{v-upper-bd} hold. Let $\{r_i\}_{i=1}^{\infty}\subset (1,\infty)$ be a sequence such that
$r_i\to\infty$ as $i\to\infty$. Then by \eqref{v-lower-bd} and \eqref{v-upper-bd}  the sequence $\{r_i\}_{i=1}^{\infty}$ has a subsequence which we may assume without loss of generality to be the sequence itself such that
\begin{equation}\label{a3-defn2}
a_3:=\lim_{i\to\infty} v(r_i)\quad\mbox{ exists  and }a_3\in [C_6,C_7].
\end{equation}
Then since \eqref{a3-defn2} implies that \eqref{fri-goes-to0} holds, we get \eqref{f-intregal-eqn5}.
Note that by \eqref{v-lower-bd}, $v\not\in L^1(\mathbb{R}^n\setminus B_1(0))$. By 
\eqref{v-upper-bd},
\begin{equation}\label{f-weight-norm-finite}
\int_1^{\infty}\rho f(\rho)\,d\rho<\infty\quad\mbox{ and }\quad
\int_1^{\infty}\frac{1}{s^{n-1}}\left(\int_1^s\rho^{n-1}f(\rho)\,d\rho\right)\,ds<\infty
\end{equation}
and by \eqref{v-lower-bd},
\begin{equation}\label{f-weight-norm-infinity}
\int_1^{\infty}\rho^{n-1}f(\rho)\,d\rho=\infty.
\end{equation}
By \eqref{c-xi-defn}, \eqref{f-intregal-eqn5}, \eqref{a3-defn2}, \eqref{f-weight-norm-finite}, \eqref{f-weight-norm-infinity}, l'Hospital's rule and an argument similar to the proof of Lemma 2.1 of \cite{Hs4},
\begin{align*}
\frac{a_3^{\frac{m}{1-m}}}{m}=&\frac{1}{m}\lim_{i\to\infty}(r_i^2f^{1-m}(r_i))^{\frac{m}{1-m}}\notag\\
=&-\frac{f(1)^{m-1}f_r(1)+\beta f(1)}{(n-2)\lim_{i\to\infty}r_i^{n-\frac{2}{1-m}}}+\beta\lim_{i\to\infty}\frac{\int_{r_i}^{\infty}\rho f(\rho)\,d\rho}{r_i^{-\frac{2m}{1-m}}}+\lim_{i\to\infty}\frac{\int_{r_i}^{\infty}\frac{\alpha-n\beta}{s^{n-1}}\left(\int_1^s\rho^{n-1}f(\rho)\,d\rho\right)\,ds}{r_i^{-\frac{2m}{1-m}}}\notag\\
=&\frac{(1-m)}{2m}\left(\beta\lim_{i\to\infty}\frac{r_i f(r_i)}{r_i^{-\frac{2m}{1-m}-1}}+(\alpha-n\beta)\lim_{i\to\infty}\frac{\frac{1}{r_i^{n-1}}\left(\int_1^{r_i}\rho^{n-1}f(\rho)\,d\rho\right)\,ds}{r_i^{-\frac{2m}{1-m}-1}}\right)\notag\\
=&\frac{(1-m)}{2m}\left(\beta a_3^{\frac{1}{1-m}}+(\alpha-n\beta)\lim_{i\to\infty}\frac{\int_1^{r_i}\rho^{n-1}f(\rho)\,d\rho}{r_i^{n-\frac{2}{1-m}}}\right)\notag\\
=&\frac{(1-m)}{2m}\left(\beta a_3^{\frac{1}{1-m}}+\frac{\alpha-n\beta}{n-\frac{2}{1-m}}\lim_{i\to\infty}\frac{r_i^{n-1}f(r_i)}{r_i^{n-\frac{2}{1-m}-1}}\right)\notag\\
=&\frac{(1-m)}{2m}\left(\beta a_3^{\frac{1}{1-m}}+\frac{\alpha-n\beta}{n-\frac{2}{1-m}}a_3^{\frac{1}{1-m}}\right)\notag\\
=&\frac{(1-m)\rho_1}{2m(n-2-nm)}a_3^{\frac{1}{1-m}}.
\end{align*}
Hence 
\begin{equation*}
a_3=\frac{2(n-2-nm)}{(1-m)\rho_1}.
\end{equation*}
Since the sequence $\{r_i\}_{i=1}^{\infty}$ is arbitrary, \eqref{f-limit-infty} hold and the theorem follows.

\end{proof}


\begin{thebibliography}{99}

\bibitem[A]{A} D.G.~Aronson, {\em The porous medium equation, CIME Lectures, in Some problems in
Nonlinear Diffusion}, Lecture Notes in Mathematics 1224, Springer-Verlag, New York, 1986.

\bibitem[DK]{DK} P.~Daskalopoulos and C.E.~Kenig, {\em Degenerate
diffusion-initial value problems and local regularity theory},
Tracts in Mathematics 1, European Mathematical Society, 2007.

\bibitem[DKS]{DKS} P.~Daskalopoulos, J.~King and N.~Sesum, {\em Extinction profile of complete non-compact solutions to the Yamabe flow}, Commun. Analysis and Geometry 27 (2019), no. 8, 1757--1798..

\bibitem[DPS]{DPS} P.~Daskalopoulos, M.~del Pino and N.~Sesum, {\em Type II ancient compact solutions to
the Yamabe flow}, J. Reine Ang. Math. 738 (2018), 1--71. 

\bibitem[DS1]{DS1} P.~Daskalopoulos and N.~Sesum, {\em On the extinction
profile of solutions to fast diffusion}, J. Reine Angew Math. 622 (2008),
95--119.

\bibitem[DS2]{DS2} P.~Daskalopoulos and N.~Sesum, {\em The classification of
locally conformally flat Yamabe solitons}, Advances in Math. 240 (2013), 346--369.

\bibitem[FMTY]{FMTY} M.~Fila, P.~Mackov\'a, J.~Takahashi and E.~Yanagida, {\em Moving singularities for nonlinear diffusion equations in two space dimensions}, J. Elliptic and Parabolic Equations  6 (2020), 155--169.


\bibitem[FVWY]{FVWY} M.~Fila, J.L.~Vazquez, M.~Winkler and E.~Yanagida, {\em Rate of convergence to
Barenblatt profiles for the fast diffusion equation}, Arch. Rational Mech. Anal. 204 (2012), no. 2, 599--625.

\bibitem[FW]{FW} M.~Fila and M.~Winkler, {\em Rate of convergence to separable solutions of the fast diffusion equatiion}, Israel J. Math. 213 (2016), 1--32.


\bibitem[Hs1]{Hs1} S.Y.~Hsu, {\em Classification of radially symmetric self-similar solutions of $u_t=\Delta\log u$ in higher dimensions}, Differential Integral Equations 18 (2005), no. 10, 1175--1192.

\bibitem[Hs2]{Hs2} S.Y.~Hsu, {\em Singular limit and exact decay rate of a nonlinear elliptic equation}, Nonlinear Anal. TMA 75 (2012), no. 7, 3443--3455.

\bibitem[Hs3]{Hs3} S.Y.~Hsu, {\em Existence and asymptotic behaviour of solutions of the very fast diffusioin equation}, Manuscripta Math. 140 (2013), nos. 3--4, 441-460.

\bibitem[Hs4]{Hs4} S.Y.~Hsu, {\em Exact decay rate of a nonlinear elliptic equation related to the Yamabe flow}, Proc. Amer. Math. Soc. 142 (2014), no. 12, 4239--4249.

\bibitem[Hu1]{Hu1} K.M.~Hui, {\em On some Dirichlet and Cauchy problems
for a singular diffusion equation}, Differential Integral Equations 15
(2002), no. 7, 769--804.

\bibitem[Hu2]{Hu2} K.M.~Hui, {\em Singular limit of solutions of the very
fast diffusion equation}, Nonlinear Anal. TMA 68 (2008), 1120--1147.

\bibitem[Hu3]{Hu3}  K.M.~Hui, {\em Asymptotic behaviour of solutions of the fast diffusion equation near its extinction time}, J. Math. Anal. Appl. 454 (2017), no. 2, 695--715. 

\bibitem[Hu4]{Hu4}  K.M.~Hui, {\em Uniqueness and time oscillating behaviour of finite points blow-up solutions of the fast diffusion equation}, Proceedings of the Royal Society of Edinburgh Section A: Mathematics  150 (2020), no. 6, 2849--2870.

\bibitem[Hu5]{Hu5} K.M.~Hui, {\em Existence of singular rotationally symmetric gradient Ricci solitons in 
higher dimensions}, Canadian Mathematical Bulletin 67 (Sept 2024), Issue 3, pp.842--859.

\bibitem[HuKj]{HuKj} K.M.~Hui and Jongmyeong Kim, {\em Asymptotic large time behaviour of singular solutions of the fast diffusion equation (preprint in preparation)}. 

\bibitem[HuK1]{HuK1} K.M.~Hui and Sunghoon Kim, {\em Large time behaviour of the higher dimensional logarithmic diffusion equation}, Proc. Royal Soc. Edinburgh 143A (2013), 817--830. 

\bibitem[HuK2]{HuK2} K.M.~Hui and Sunghoon Kim, {\em Existence and large time behaviour of finite points blow-up solutions of the fast diffusion equation}, Calc. Var. Partial Differential Equations 57 (2018), no. 5, Paper No. 112, 39 pp. 

\bibitem[HuK3]{HuK3} K.M.~Hui and Sunghoon Kim, {\em Singular limits and properties of solutions of some degenerate elliptic and parabolic equations}, Proceedings Royal Soc. Edinburgh 149 (2019), 353--385.


\bibitem[JX]{JX} T.~Jin and J.~Xiong, {\em Singular extinction profile of solutions to some fast diffusion equations}, J. Functional Analysis 283 (2022), issue 7, 109595.

\bibitem[PS]{PS} M.~Del Pino and M.~S\'aez, {\em On the extinction profile
for solutions of $u_t=\Delta u^{(N-2)/(N+2)}$}, Indiana Univ. Math. J. 50
(2001), no. 1, 611--628.

\bibitem[SWZ]{SWZ} Y.~Sire, J.~Wei and Y.~Zheng, {\em Extinction behavior for the fast diffusion equations with critical exponent and Dirichlet boundary condition}, J. London Math. Soc. 106 (2022), issue 2, 855--898. 

\bibitem[TY]{TY} J.~Takahashi and H.~Yamamoto, {\em Infinite-time incompleteness of noncompact Yamabe flow},  Calc. Var. and PDE 61, Article number: 212 (2022).   


\bibitem[V]{V} J.L.~Vazquez, {\em Smoothing and decay estimates for nonlinear
diffusion equations}, Oxford Lecture Series in Mathematics and its Applications
33, Oxford University Press, Oxford, 2006.

\end{thebibliography}
\end{document}